\documentclass[11pt,reqno,twoside]{amsart}
\usepackage{graphicx}
\UseRawInputEncoding 
\usepackage{subfigure}
\usepackage{amsmath,mathdots,amssymb,latexsym,amsthm,mathrsfs,epsfig}
\usepackage{multirow}
\usepackage{breqn}
\usepackage{float}
\usepackage{cite}
\usepackage{enumitem}
\usepackage{environ}
\usepackage{mathdots}
\usepackage{mathtools}
\NewEnviron{myequation}{
	\begin{equation}
		\scalebox{1.1}{$\BODY$}
	\end{equation}
}
\usepackage{color}
\usepackage{soul}
\renewcommand{\baselinestretch}{1.2}
\makeatletter
\newcommand{\single}{\let\CS=\@currsize\renewcommand{\baselinestretch}{1.1}\tiny\CS}
\newcommand{\singb}{\let\CS=\@currsize\renewcommand{\baselinestretch}{1}\tiny\CS}
\newcommand{\singa}{\let\CS=\@currsize\renewcommand{\baselinestretch}{1.2}\tiny\CS}
\newcommand{\oneandahalfspacing}{\let\CS=\@currsize\renewcommand{\baselinestretch}{1.5}\tiny\CS}
\newcommand{\singlespacing}{\let\CS=\@currsize\renewcommand{\baselinestretch}{1.6}\large\CS}
\newcommand{\bc}{\begin{center}}
	\newcommand{\ec}{\end{center}}
\newcommand{\be}{\begin{eqnarray}}
	\newcommand{\ee}{\end{eqnarray}}

\makeatletter
\newcommand{\beano}{\begin{eqnarray*}}
	\newcommand{\eeano}{\end{eqnarray*}}

\newcommand{\ba}{\begin{array}}
	\newcommand{\ea}{\end{array}}

\makeatletter

\DeclareMathOperator{\D}{D}

\DeclareMathOperator{\Te}{T}
\DeclareMathOperator{\Fe}{F}
\DeclareMathOperator{\GL}{GL}
\DeclareMathOperator{\Sp}{Sp}
\DeclareMathOperator{\Rad}{Rad}
\DeclareMathOperator{\Cusp}{Cusp}
\DeclareMathOperator{\Ne}{N}

\DeclareMathOperator{\ST}{St}

\newcommand{\eval}[2][\right]{_lax
	\ifx#1\right\relax \left.\fi#2#1\rvert}
\newcommand{\diag}{\operatorname{diag}}
\newcommand{\Hom}{\operatorname{Hom}}

\newcommand{\Dim}{\operatorname{dim}}
\newcommand{\Irr}{\operatorname{Irr}}
\newcommand{\re}{\operatorname{r}}
\newcommand{\Ind}{\operatorname{Ind}}
\newcommand{\ind}{\operatorname{ind}}
\newcommand{\pr}{\operatorname{pr}}
\newcommand{\Supp}{\operatorname{Supp}}
\newcommand{\Sgm}{\operatorname{Sgm}}

\newtheorem{rem}{Remark}[section]
\newtheorem{prop}{Proposition}[section]
\newtheorem{cor}{Corollary}[section]
\newtheorem{defi}{Definition}[section]
\newtheorem{exa}{Example}[section]

\newtheorem{theorem}{Theorem}[section]

\newtheorem{lemma}[theorem]{Lemma}

\numberwithin{equation}{section}
\allowdisplaybreaks
\begin{document}
	
	\title[SYMPLECTIC MODEL FOR LADDER AND UNITARY REPRESENTATIONS]{SYMPLECTIC MODEL FOR LADDER AND UNITARY REPRESENTATIONS}
	

\author[Hariom Sharma]{Hariom Sharma}
	\address{Department of Mathematics, Indian Institute of Technology Roorkee, Uttarakhand, 247667, Bharat(India)}
	\email{hariom\textunderscore s@ma.iitr.ac.in}

	
	\author[Mahendra Kumar Verma]{Mahendra Kumar Verma}
	\address{Department of Mathematics, Indian Institute of Technology Roorkee, Uttarakhand, 247667, Bharat(India)}
	\email{mahendraverma@ma.iitr.ac.in}


	\begin{abstract} Let $\D$ denote a quaternion division algebra over a non-archimedean local field $\Fe$ with characteristic zero. Let $\Sp_n(\D)$ be the unique non-split inner form of the symplectic group $\Sp_{2n}(\Fe)$. An irreducible admissible representation $(\pi, V)$ of $\GL_{n}(\D)$ is said to have a symplectic model (or said to be $\Sp_n(\D)$-distinguished) if there exists a linear functional $\phi$ on $V$ such that $\phi(\pi(h)v) = \phi(v)$ for all $v \in V$ and $h \in \Sp_n(\D)$.
	This article classifies those ladder representations of $\GL_n(\D)$ that possess a symplectic model (i.e., those representations that are $\Sp_n(\D)$-distinguished).  Recently, Prasad \cite[Conjecture 7.1~(1)]{Verma} conjectured that non-supercuspidal discrete series representations of $\GL_n(\D)$ do not admit a symplectic model. We confirm this for the Steinberg representations, which serve as canonical examples of discrete series representations. Furthermore, we demonstrate the hereditary nature of the symplectic model for induced representations derived from finite-length representations. In addition, we prove a part of Prasad's conjecture \cite[Conjecture 7.1~(2)]{Verma}, which provides a family of irreducible unitary representations, all equipped with a symplectic model.
	\end{abstract}
	
	\maketitle
	

	\setcounter{page}{1}
	\section{INTRODUCTION}
	Let $\mathcal{G}$ be an $l$-group and $\mathcal{H}$ be a closed subgroup of $\mathcal{G}$. A smooth admissible complex representation $(\pi,V)$ of $\mathcal{G}$  is said to be $\mathcal{H}$ -distinguished if there exists a linear functional $\phi$ on $V$ such that $\phi(\pi(h)v) = \phi(v)$ for all $v \in V$ and $h \in \mathcal{H}$.
	
     Distinguished representations play a crucial role in studying automorphic forms and representations, particularly in the context of the Langlands program. They help in understanding the connections between Galois representations and automorphic representations. In the context of the trace formula, distinguished representations help to analyze periods of automorphic forms over subgroups, providing insights into special values of $L$-functions and other arithmetic properties. These representations are also fundamental in harmonic analysis on $p$-adic groups and their homogeneous spaces. They allow for the decomposition of $L^2$ spaces of functions on these groups, leading to significant insights in Fourier analysis over local fields. This interplay between harmonic analysis and representation theory extends further into theoretical physics. 
     
     In theoretical physics, particularly in quantum mechanics and the study of quantum groups, distinguished representations help in the formulation and understanding of symmetries and invariants. They provide a mathematical framework for studying the symmetry properties of physical systems. 
     
     Consider a non-archimedean local field $\Fe$ of characteristic zero, and let $\D$ be the unique quaternion division algebra over $\Fe$. The reduced trace and reduced norm maps on $\D$ are denoted by $\Te_{\D/\Fe}$ and $\Ne_{\D/\Fe}$, respectively. Define an involution $\psi$ on $\D$ by $x \mapsto \overline{x} = \Te_{\D/\Fe}(x) - x$.
     
      For a fixed $n \in \mathbb{N}$, let $G = \GL_n(\D)$ throughout this article. The subgroup $H = \Sp_n(\D) \subset G$ represents the unique non-split inner form of the symplectic group $\Sp_{2n}(\Fe)$, defined as \linebreak
     $H = \{ A \in G~ | ~ A ~J ~ ^t\bar{A} = J \},$
     where $^t\bar{A} = (\bar{a}_{ji})$ for $A = (a_{ij})$ and 
     $$ J = J_n = \begin{pmatrix}
     	&  &  &  1 \\
     	&  &  1  & \\
     	&  \reflectbox{$\ddots$}  &  & \\
     	1  &  &  &  & 
     \end{pmatrix}.$$
 Let $\theta$ be the involution on $G$ defined by $g \mapsto {J}~^t\bar{g}^{-1} J$.
 Then the group of fixed points in $G$ under this involution $\theta$ is $G^{\theta} = H$. Finally, we set $G_n = \GL_n(\D)$, $G_n' = \GL_n(\Fe)$, and $H_n = \Sp_n(\D)$. A smooth admissible complex representation $(\pi,V)$ of $G_n$ is said to have a symplectic model if it is $H_n$-distinguished.
 
     Prasad \cite{prasad1993decomposition} and Venketasubramanian \cite{Venket} explored the $G_{n-1}'$-distinguished representations of $G_{n}'$, whereas Yang \cite{yang2022linear} examined the unitary representations of $G_{2n}'$, which are distinguished by $G_n' \times G_n'$.
      The primary aim of this article is to classify those ladder and irreducible unitary representations of $G_n$ that have a symplectic model (i.e., those representations that are $H_n$-distinguished).
	
	We get motivation for this problem from the findings of  Mitra et al., \cite{mitra2017klyachko} regarding distinguished ladder representations of $G_n'$ with respect to Klyachko models $\mathcal{M}_{n,k}$, $0 \leq k \leq [n/2]$, (see \cite[Section $1$]{mitra2017klyachko} for definition of Klyachko models). Note that the symplectic model is a specific instance of the Klyachko model.
      In the case of finite fields $\Fe$, these models were founded by Klyachko\cite{AA84} and are entirely separate when considering different indices. Furthermore, there are no multiple components in them, and each irreducible representation of $G_n'$ is present in at least one of the Klyachko models. 
      
     When $\Fe$ is a non-archimedean local field, Heumos and Rallis \cite{HR90} turn their attention to the study of these models, where they observed the disjointness of the different models for the unitary representations and proved that all irreducible unitary representations of $G_n'$ admit a Klyachko model for $n \leq 4$. Furthermore, they proved the uniqueness of the symplectic model, i.e., for any irreducible admissible representation $\pi$ of $G_{2n}'$, 
     $$\Dim \Hom_{\Sp_{2n}(\Fe)} (\pi|_{\Sp_{2n}(\Fe)} , \mathbb{C}) \leq 1.$$
     
      Offen and Sayag \cite{OSE09, OSE08} later proved the uniqueness and disjointness of Klyachko models for the general $n$ case and showed that every irreducible unitary representation admits a Klyachko model. Moreover, they classified all the irreducible unitary representations with a symplectic model in \cite{OSE07}. 
      The irreducible admissible representations of $G_{2n}'$ with a symplectic model are studied by Mitra \cite{mitra2014representations} for small-rank cases. Mitra et al., \cite{mitra2017klyachko} later classified those ladder representations of $G_{n}'$ that are distinguished with respect to Klyachko models. 
      
      In the context of quaternion division algebra, Verma \cite{Verma} proved the uniqueness of the symplectic model and classified the irreducible admissible representations of $G_1$ and $G_2$ with a symplectic model. Recently, authors \cite{sharma2023symplectic} extended this study to the irreducible admissible representations of $G_{n}$ for $n=3$ and $4$. We now aim to generalize this investigation to the irreducible admissible representations of $G_{n}$ with a symplectic model for arbitrary $n$. This involves examining various classes of irreducible admissible representations of $G_{n}$, mainly focusing on the ladder representations introduced by Lapid and Minguez \cite[Section 3]{lapid2014determinantal} and the irreducible unitary representations described by Tadic \cite[Section 6]{tadic1990induced}. Within the global scenario, Verma \cite{verma2018distinguished} investigated the symplectic period for the pair $(G_2,H_2)$, and in the future, we will focus on finding the global symplectic period for the pair $(G_{2n}, H_{2n})$. 
  \subsection{Main Results} To state the main theorem, which describes ladder representations of $G_n$ with a symplectic model, we need to introduce some notations (we refer the reader to Section $2$ for clarification on any unfamiliar notation). 
  
  Let $\nu = |\Ne_{\D/\Fe}(\cdot)|_{\Fe}$ denote the character of $G_r$ given by the absolute value of the reduced norm. For $\rho \in \Cusp (G_r)$, define $\nu_{\rho} = \nu^{s(\rho)}$, a character of $G_r$, where $s(\rho)$ is a divisor of $2$ satisfying $(s(\rho),r) = 1$ (for more details, see \cite[Section $2$]{tadic1990induced}), i.e., $\nu_{\rho}$ equals either $\nu$ or $\nu^2$. Note that the character $\nu_{\rho}$ of $G_r$ depends on the choice of representation $\rho$. In particular, when $r$ is even, $\nu_{\rho}$ coincides with $\nu$. In the context of odd $r$, we consider those supercuspidal representations $\rho$ for which $\nu_{\rho}$ corresponds to $\nu$ throughout this article. 
  
  Let $\rho \in \Cusp$. Then we define 
  $\Delta = [a,b]_{\rho} =  [\nu_{\rho}^a\rho,\nu_{\rho}^b\rho]=\{\nu_{\rho}^a\rho, \nu_{\rho}^{a+1}\rho, \ldots,\nu_{\rho}^b\rho\}$
  to be a segment corresponding to $\rho$, where $a \leq b$ are integers. By the definition of  the segment $\Delta = [a,b]_{\rho}$, the begining $b(\Delta)$ is $a$, the end $e(\Delta)$ is $b$, and the length $l(\Delta)$ is $b-a+1$. 
  
  Let $\Sgm$ denote the set of all segments.  
  For a segment $\Delta = [a,b]_{\rho}$ in $\Sgm$, we associate the representation $\nu_{\rho}^a\rho \times \cdots \times \nu_{\rho}^b\rho$ which has a unique irreducible subrepresentation denoted by $Z(\Delta)$ and a unique irreducible quotient denoted by $Q(\Delta)$. For convenience, let $Z(\Delta)$ (resp., $Q(\Delta)$) be the trivial representation of $G_0$ if $b = a-1$ and $0$ if $b < a-1$. Notably, $Q(\Delta)$ is essentially square-integrable. The map $\Delta \mapsto Q(\Delta)$ establishes a bijection between $\Sgm$ and the subset of essentially square-integrable representations in $\Irr$.
  
  Let $\mathcal{O}$ be the set of multisets of segments in $\Sgm$. We say a multiset $\mathfrak{m} = \{\Delta_1,\ldots,\Delta_t\} \in \mathcal{O}$ is on the supercuspidal line of $\rho \in \Cusp$ if $\Delta_i \subset \{\nu_{\rho}^k \rho\}_{k \in \mathbb{Z}}$, for each $1 \leq i \leq t$. For $\rho \in \Cusp$, let $\mathcal{O}_{\rho} \subset \mathcal{O}$ be the set of multisets of segments in $\Sgm$, which are on the supercuspidal line of $\rho$. A multiset $\mathfrak{m} \in \mathcal{O}_{\rho}$ is called \textbf{Speh type} if $\mathfrak{m} = \mathfrak{m'} + \nu \mathfrak{m'}$ for some $\mathfrak{m'} \in \mathcal{O}_{\rho}$.
  
  Let $\nu_{\rho} \Delta = [a+1,b+1]_{\rho}$, and let
  $\Delta_1 = [a_1,b_1]_{\rho_1}$ and $\Delta_2 = [a_2,b_2]_{\rho_2}$ be two segments in $\Sgm$. Then $\Delta_1$ precedes $\Delta_2$ (or $\Delta_1 \prec \Delta_2$) if both $\Delta_1$ and $\Delta_2$ are contained in some single $\rho$-supercuspidal line, $a_1<a_2$, $b_1<b_2$, and $b_1 \geq a_2 - 1$. Furthermore, the segments $\Delta_1$ and $\Delta_2$ are called linked if either $\Delta_1$ precedes $\Delta_2$ or $\Delta_2$ precedes $\Delta_1$.
  
  An order $\mathfrak{m} = \{\Delta_1,\ldots,\Delta_t\} \in \mathcal{O}$ on a multiset $\mathfrak{m}$ is of standard form if $\Delta_i \not\prec \Delta_j$ for all $i < j$.
  For an ordered multiset $\mathfrak{m} = \{\Delta_1,\ldots,\Delta_k\}$ in standard form, we attach the representation $\pi(\mathfrak{m}) = Z(\Delta_1) \times\cdots \times Z(\Delta_k)$ which has a unique irreducible subrepresentation denoted by $Z(\mathfrak{m})$. Analogously, we define $\lambda(\mathfrak{m}) = Q(\Delta_1) \times\cdots \times Q(\Delta_k)$ which has a unique irreducible quotient denoted by $Q(\mathfrak{m})$. Note that $\pi(\mathfrak{m})$ and $\lambda (\mathfrak{m})$ are independent of the choice of order of standard form. According to Zelevinsky classification \cite{minguez2013representations}, the map $(\mathfrak{m} \mapsto Z(\mathfrak{m})) \colon \mathcal{O} \rightarrow \Irr$ is a bijection, and similarly Langlands classification \cite{tadic1990induced} establishes that the map $(\mathfrak{m} \mapsto Q(\mathfrak{m})) \colon  \mathcal{O} \rightarrow \Irr$ is also a bijection. Badulescu and  Renard \cite{badulescu2007zelevinsky} proved that the algorithm of Moeglin and Waldspurger for computing the dual $\mathfrak{m}^t$ (as defined by Zelevinsky) for the multisegment $\mathfrak{m}$ also works for the inner forms of $G_n'$. Therefore, for any $\mathfrak{m} \in \mathcal{O}$ there exists a unique $\mathfrak{m}^t \in \mathcal{O}$ such that $Z(\mathfrak{m}) = Q(\mathfrak{m}^t)$. Thus, the map $\mathfrak{m} \mapsto \mathfrak{m}^t$ is an involution on $\mathcal{O}$.
  \begin{defi} 	Let $\rho \in \Cusp$. The multiset $\mathfrak{m} = \{\Delta_1, \Delta_2, \ldots, \Delta_t\} \in \mathcal{O}_{\rho}$ is defined as a ladder if  $$b(\Delta_1) > b(\Delta_2) > \cdots > b(\Delta_t) ~\text{and}~ e(\Delta_1) > e(\Delta_2) > \cdots > e(\Delta_t) .$$
  	Given a ladder $\mathfrak{m} \in \mathcal{O}_{\rho}$, a representation $\pi \in \Irr$ is called a ladder if $\pi = Q(\mathfrak{m})$.
  \end{defi}
   
      We now summarize the main results of this article.
      \begin{theorem}\label{1}
      	 Let $\mathfrak{m} = \{\Delta_1,\ldots,\Delta_t\} \in \mathcal{O}_{\rho}$ be a ladder, where $\rho \in \Cusp$ does not have a symplectic model.
      	Then the following are equivalent:
      	\begin{enumerate}
      		\item[\upshape(1)] $Q(\mathfrak{m})$ admits a symplectic model;
      		\item[\upshape(2)] $\mathfrak{m} = \mathfrak{m'} + \nu \mathfrak{m'}$ for some $\mathfrak{m'} \in \mathcal{O}_{\rho}$;
      		\item[\upshape(3)] $\Delta_{2i-1} = \nu \Delta_{2i}$, where $1 \leq i \leq t/2$, with $t$ being even.
      	\end{enumerate}
      \end{theorem}
  \begin{cor}
  	Let $\mathfrak{m} \in \mathcal{O}_{\rho}$ be a ladder, where $\rho \in \Cusp$ does not have a symplectic model. Then \linebreak$Q(\mathfrak{m}) \in \Irr(G_4)$ has a symplectic model if and only if $\mathfrak{m}$ is one of the following:  
  	\begin{itemize}
  	\item[\upshape(1)] $\mathfrak{m} = \{[a+1]_{\rho}, [a]_{\rho}\}$, where $\rho \in \Cusp(G_2)$.
  	\item[\upshape(2)] $\mathfrak{m} \in \{\{[a+1 , a+2]_{\rho},[a, a+1]_{\rho}\}, \{[a+2]_{\rho},[a+1]_{\rho},[a]_{\rho},[a-1]_{\rho}\}\}$, where $\rho \in \Cusp(G_1)$, with $\Dim(\rho)>1$.
  \end{itemize}
  \end{cor}

   Note that the above corollary also follows directly from Theorem $1.3$ in \cite{sharma2023symplectic}, indicating that Theorem \ref{1} extends this theorem  for the aforementioned class of ladder representations.

    Using Moeglin and Waldspurger's algorithm, Lapid and Minguez \cite{lapid2014determinantal} observed that $\mathfrak{m}$ being a ladder implies $\mathfrak{m}^t$ being a ladder, and vice versa. Therefore, the set of ladder representations remains invariant under Zelevinsky involution, allowing us to identify the ladders $\mathfrak{m} \in \mathcal{O}_{\rho}$ such that $Z(\mathfrak{m})$ possesses a symplectic model.  It should be noted that in a ladder $\mathfrak{m} = \{\Delta_1,\ldots,\Delta_t\}$, the intersection $\Delta_i \cap \Delta_{i+1}$ is either void or forms a segment. Thus, the length $l(\Delta_i \cap \Delta_{i+1})$ is meaningful for every $1 \leq i \leq t-1$. Assume $\Lambda_i = \Delta_i \cap \Delta_{i+1}$ and $\Psi_i = \Delta_i \cup \Delta_{i+1}$ for each $1 \leq i \leq t-1$. Then we have the following elementary observation related to Speh multisets in \cite[$\S$$10.2.2$]{mitra2017klyachko}.
   \begin{rem}\label{2}
   	   	Let $\rho \in \Cusp$. Let $\mathfrak{m} \in \mathcal{O}_{\rho}$ be a ladder with $\mathfrak{m}^t = \{\Delta_1,\ldots,\Delta_t\}$ ordered as a ladder. Then $\mathfrak{m} = \mathfrak{m'} + \nu \mathfrak{m'}$ for some $\mathfrak{m'} \in \mathcal{O}_{\rho}$ if and only if for every $1 \leq i \leq t$, $l(\Delta_i)$ is even  and for every $1 \leq i \leq t-1$, $l(\Lambda_i)$ is odd  such that $\Psi_i$ forms a segment.
   \end{rem}    
   Using Theorem \ref{1} with Remark \ref{2}, we get the following classification of ladder representation having a symplectic model in terms of Zelevinsky's classification.
   \begin{cor}\label{3}
   	With the notation in Theorem \ref{1}, $Z(\mathfrak{m})$ admits a symplectic model if and only if for every $1 \leq i \leq t$, $l(\Delta_i)$ is even  and for every $1 \leq i \leq t-1$, $l(\Lambda_i)$ is odd such that $\Psi_i$ forms a segment. 
    \end{cor}
      Theorem \ref{1} does not deal with ladder representations corresponding to an irreducible supercuspidal representation with a symplectic model; therefore, it remains an open problem. 
       In particular, one can also study irreducible supercuspidal representations of $G_n$, which admit a symplectic model. D. Prasad stated a conjecture \cite[Conjecture 7.1~(1)]{Verma} regarding supercuspidal representations. In \cite{Verma}, Conjecture 7.1 and Proposition 7.3 jointly suggest that the supercuspidal representations of $G_n$ do not have a symplectic model for even values of $n$.
       However, D. Prasad constructed some examples of supercuspidal representations of $G_n$ having a symplectic model for odd values of $n$. Therefore, examining the presence of a symplectic model within supercuspidal representations of $G_n$ is a fascinating problem for odd values of $n$. 
       Conjecture $7.1~(1)$ in \cite{Verma} also predicts that non-supercuspidal discrete series representations of $G_n$ do not admit a symplectic model. We verify this for the Steinberg representations of $G_n$.
       \begin{prop}\label{301}
       	The Steinberg representations $\ST_n(\chi)$ of $G_n$ do not admit a symplectic model.
       \end{prop}
       
        These representations serve as canonical examples of discrete series representations characterized by their construction via induced representations and their irreducibility and unitarity. Understanding these representations provides insight into the broader landscape of admissible and tempered representations. We also demonstrate the hereditary nature of the symplectic model, as observed in \cite[Theorems $1.2$ and $1.3$]{sharma2023symplectic}.
       \begin{theorem}\label{48}
If $\sigma_i \in \Pi(G_{n_i})$ admits a symplectic model for each $1 \leq i \leq k$, then $\sigma_1 \times \cdots \times \sigma_k$ admits a symplectic model.       	
       \end{theorem}

      In the framework of unitary representations, D. Prasad proposed a conjecture \cite[Conjecture 7.1~(2)]{Verma} describing those irreducible admissible unitary representations $\pi$ of $G_n$ which have a symplectic model. However, it was observed by authors \cite{sharma2023symplectic} that some unitary representations are not included in that list. As a result, Conjecture 7.1~(2) in \cite{Verma} needs modification. To state the conjecture, we recall Tadic's classification \cite[Section $6$]{tadic1990induced} of the unitary dual of $G$.

       Let $m \in \mathbb{N}$ and $\delta$ be an essentially square-integrable representation of $G_r$. Then the Speh representation $\Sp(\delta,m)$ is called the Langlands quotient of the representation
       $$\nu_{\delta}^{\frac{m-1}{2}}\delta \times \nu_{\delta}^{\frac{m-3}{2}}\delta \times \cdots \times \nu_{\delta}^{\frac{-(m-1)}{2}}\delta.$$
       For a representation $\tau$ of $G_p$ and $\alpha \in \mathbb{R}$, we let $\pi(\tau, \alpha) = \nu_{\delta}^{\alpha} \tau \times \nu_{\delta}^{-\alpha} \tau$. Further, define 
     $$\mathcal{B} = \{\Sp(\delta,n), \pi(\Sp(\delta,n),\alpha) \colon \delta \text{-unitary essentially square-integrable},~ n \in \mathbb{N}, ~\text{and}~|\alpha|< 1/2\}.$$ 
     According to \cite[Section $6$]{tadic1990induced}, for a representation $\pi$ to be unitarizable, it must be expressible as $\eta_1 \times \eta_2 \times \cdots \times \eta_s$, where $\eta_i \in \mathcal{B}$ for every $1 \leq i \leq s$. 
      Furthermore, this decomposition is unique up to permutation. It is referred to as a Tadic decomposition of $\pi$. We have now modified Prasad's conjecture and have proven the following: 
     \begin{theorem}\label{4}
     	The unitary representations which are of the form 
     	$$\pi = \rho_1 \times \cdots \times \rho_u \times \lambda_{1} \times \cdots \times \lambda_{v} \times \tau_{1} \times \cdots \times \tau_{x} \times \sigma_{1} \times \cdots \times \sigma_{y}$$
     admit a symplectic model where 
     	\begin{itemize}
     		\item[\upshape{(1)}] $\rho_i$ are Speh representations $\Sp(\delta_i,2m_i)$, with $\delta_i$ being the essentially square-integrable,
     		\item[\upshape{(2)}] $\lambda_i$ are complementary series representations $\pi(\Sp(\delta_i,2r_i), \alpha_i)$ with $|\alpha_i|<1/2$,
     		\item[\upshape{(3)}] $\tau_i$ are supercuspidal representations having a symplectic model, 
     		\item[\upshape{(4)}]  $\sigma_i$ are principal series representations $\chi_2 \ST_2 \times \chi_1$ of $G_3$, with $\chi_2 \ST_2$ being the Steinberg representation of $G_2$ and a quotient of $\nu^{-1}\chi_1 \times \nu \chi_1$.  
     	\end{itemize}
     \end{theorem}
     We establish Theorem \ref{4} only for essentially square-integrable representations $\delta$ associated with a supercupidal representation $\rho$ without symplectic model. This theorem provides a family of unitary representations with a symplectic model, which is larger than the family mentioned in \cite[Conjecture 7.1~(2)]{Verma}, and does not claim to have all irreducible unitary representations with a symplectic model. Note that the situation in $G_n$  for the family of unitary representations having a symplectic model is different from $G_n'$, as $G_n$ has a broader family of such representations.
     
     \subsection{About the proofs} In problems of distinction, it is more convenient to consider a representation $\pi$ as a quotient. The absence of symplectic models in Steinberg's representations is demonstrated through their representation as quotients of induced representations lacking such models. Furthermore, we establish the hereditary property of symplectic models in induced representations by initially constructing a non-zero linear form on the open orbit. Extending this linear form to the induced representation space utilizes a result from Blanc and Delorme, thereby completing the proof.
     
     In the context of Theorem \ref{1}, the focus shifts to examining the existence of a symplectic model for induced representation 
     $\lambda(\mathfrak{m})$ because $Q(\mathfrak{m})$ emerges as a quotient of $\lambda(\mathfrak{m})$. To investigate the presence of symplectic models within induced representations, we employ Bernstein and Zelevinsky's geometric lemma. This lemma reduces the question of a symplectic model's existence in an induced representation to distinguishing certain Jacquet modules of the inducing data.
     This technique enables us to reduce the problem of the symplectic model for $\lambda(\mathfrak{m})$ to a specific type of multisets of integer segments.
     
     It is easy to see that the existence of a symplectic model for $Q(\mathfrak{m})$ reduces the study of a symplectic model for $\lambda(\mathfrak{m})$ and its maximal proper subrepresentation $\mathcal{P}$. Therefore, the non-distinguished nature of $\mathcal{P}$, along with the existence of a symplectic model for $\lambda(\mathfrak{m})$ by using the above technique, concludes the proof of Theorem \ref{1}. Using Theorem \ref{1} and the hereditary nature of the symplectic model,
     we prove Theorem \ref{4}, which offers a family of irreducible unitary representations of $G_n$, all endowed with a symplectic model.  
     
     \subsection{Organization} The article is organized as follows. In Section $2$, we set up the notations and some preliminaries for $G_n$ and recall some general facts related to the symplectic model. For a parabolic subgroup $P$ of $G_n$, we study $P$-orbits in $G_n/H_n$ in Section $3$. 
     Section $4$ introduces Bernstein and Zelevinsky's geometric lemma and proves the non-distinguished nature of Steinberg representations and the hereditary nature of the symplectic model. 
     
      Section $5$ explicates some notations involving the multisets of integer segments and shows a connection between the multisets of integer segments and the symplectic model. 
    We demonstrate Theorem \ref{1} and Theorem \ref{4} in Section $6$.

	\section{NOTATION AND PRELIMINARIES}\label{108} 
	 \subsection{Generalities}\label{106}
	 Let $\mathcal{H}$ be a closed subgroup of an $l$-group $\mathcal{G}$ equipped with a complex-valued smooth representation $(\sigma,\mathcal{W})$. Let $\delta_{\mathcal{K}}$ indicate the modulus function of an $l$-group $\mathcal{K}$. We denote $\Ind_{\mathcal{H}}^{\mathcal{G}}(\sigma)$ as the normalized induced representation of $\mathcal{G}$, the space of functions $f$ from $\mathcal{G}$ to $\mathcal{W}$ satisfying the condition
	 $$f(hg) = (\delta_{\mathcal{H}}^{1/2}\delta_{\mathcal{G}}^{-1/2})(h)\sigma(h)f(g), ~ h \in \mathcal{H}, ~ g \in \mathcal{G},$$
	 where $f$ is right-invariant under some open compact subgroup of $\mathcal{G}$. The group $\mathcal{G}$ acts on this space of functions by right translation. The representation of $\mathcal{G}$ on the subspace of functions with compact support modulo $\mathcal{H}$ is denoted by $\ind_{\mathcal{H}}^{\mathcal{G}}(\sigma)$.
	 
	 Let $\mathcal{P} = \mathcal{M} \mathcal{N}$ be an $l$-group where $\mathcal{M}, \mathcal{N}$ are closed subgroups of $\mathcal{P}$. Let $\mathcal{N}$ be a normal subgroup, which is the union of its compact open subgroups. Consider $(\pi, V)$ as a smooth representation of $\mathcal{P}$. Let $V_{\mathcal{N}}=V/V(\mathcal{N})$ where the subspace  $V(\mathcal{N})$ is the linear span of vectors $\{\pi(u)x-x \colon u \in \mathcal{N}, x \in V\}$. The representation $\pi$ defines a representation $\pi_{\mathcal{N}}$ of $\mathcal{P}$ in the space $V_{\mathcal{N}}$. The representation $\delta_{\mathcal{P}}^{-1/2}\delta_{\mathcal{M}}^{1/2}\pi_{\mathcal{N}}$ is called the normalized Jacquet module of $\pi$ and its restriction to $\mathcal{M}$ is denoted by $\re_{\mathcal{N}}(\pi)$. Thus to a given representation $\pi$ of $\mathcal{P}$, the functor $\re_{\mathcal{N}}$ associates a representation $\re_{\mathcal{N}}(\pi)$ of $\mathcal{M}$ to $\pi$. We call this functor the Jacquet functor (for more details, see \cite[Section $1$]{BZ}).
	 
	 Define $\Pi({\mathcal{G}})$ as the category of complex, smooth, admissible representations of $\mathcal{G}$ of finite length and $\Irr(\mathcal{G})$ the class of irreducible representations in $\Pi(\mathcal{G})$.
	  Let $\widetilde{\pi}$ denote the contragredient of a representation of $\pi \in \Pi(\mathcal{G})$.

  \subsection{Notations for $G_n$} 
  Fix a minimal parabolic subgroup $P_0$ of $G$ and a maximal split torus $T$ of G contained in $P_0$. Let  $M_0 = C_G(T)$ be the centralizer of $T$ in $G$ and  $U_0$ denote the unipotent radical of $P_0$. Then $P_0 = M_0 U_0$ is a Levi decomposition. 
  A parabolic subgroup $P$ is called semi-standard if it contains $T$ and standard if it contains $P_0$. Let $P$ be a standard parabolic subgroup of $G$ with unique Levi decomposition $P = MU$, where $M$ and $U$ are called standard Levi subgroup and unipotent radical of $G$, respectively. Let $N_{G,\theta}(M) = \{g \in G \colon M = g \theta(M) g^{-1}\}$.

  For a partition $\alpha = (n_1,n_2,\ldots,n_k)$ of $n$, we define $P_{\alpha} = M_{\alpha} \ltimes U_{\alpha}$  the semi-standard parabolic subgroup of $G$ consisting of block upper triangular matrices with semi-standard Levi subgroup $M_{\alpha} = \{\diag(g_1,\ldots,g_k) \colon g_i \in G_{n_i},~i = 1,\ldots,k \} \simeq G_{n_1} \times \cdots \times G_{n_k}$ and unipotent radical $U_{\alpha}$. The identity element of any group is denoted by  $e$. Specifically, $e = I_n$ represents an $n \times n$ identity matrix. 
  
  Let $W_M = N_M(T)/M_0$ be the Weyl group of $M$ with respect to $T$. In particular, $W_G = N_G(T) / M_0$ is the Weyl group of $G$ with respect to $T$, which is isomorphic to the permutation group $S_n$ of $n$ elements. We identify $W_G$ with the subgroup $W$ of permutation matrices in $G$ and assume $w_n = (\delta_{i,n+1-j}) \in W,$ the longest Weyl element. Note that the involution $\theta$ preserves both $T$ and $M_0$. In particular, $\theta$ induces an involution on $W$ given by $w \mapsto w_{n} w w_n^{-1}$. 
  Let $\mathfrak{J}_0(\theta) = \{w \in W \colon w \theta(w) = e\}$
  be the set of twisted involutions in $W$.

  \subsection{Bruhat decomposition}\label{bruhat}
  Consider $P_1 = M_1 \ltimes U_1$ and $P_2 = M_2 \ltimes U_2$ as standard parabolic subgroups of $G$ with their respective standard Levi decompositions. Let $_{M_1}W_{M_2}$ denote the set of all $w \in W$ that are simultaneously left $W_{M_1}$-reduced and right $W_{M_2}$-reduced. This set serves as a collection of representatives for the double cosets $W_{M_1}\backslash W/W_{M_2}$, comprising solely the elements of minimal length within their respective double cosets. According to the Bruhat decomposition, a bijective mapping $(P_1gP_2 \mapsto w) \colon P_1\backslash G/P_2 \rightarrow~ _{M_1}W_{M_2}$ exists whenever $P_1gP_2 = P_1wP_2$.

  For every $w \in {_{M_1}W}_{M_2}$, we have 
  $P_1 \cap w P_2 w^{-1} = (M_1 \cap w P_2 w^{-1})(U_1 \cap w P_2 w^{-1}).$
  Assume $$P_1(w) := \pr_{M_1}(P_1 \cap w P_2 w^{-1}) = M_1 \cap w P_2 w^{-1}.$$
  Then $P_1(w)$ is a standard parabolic subgroup of $M_1$ (with respect to $M_1 \cap P_0$) with standard Levi decomposition
  $P_1(w) = M_1(w) \ltimes U_1(w),$
  where $M_1(w) = M_1 \cap w M_2 w^{-1}$ and $U_1(w) =  M_1 \cap w U_2 w^{-1}$.
	  \subsection{Representations of $G_n$}\label{303}
	  Let $\Pi$ be the disjoint union of $\Pi(G_n)$ for all $n \in \mathbb{Z}_{\geq 0}$. Denote by $\Irr$ (resp.~$\Irr(G_n)$) the subset of irreducible representations in $\Pi$ (resp.~$\Pi(G_n)$) and $\Cusp$ (resp.~$\Cusp(G_n)$ the subset of supercuspidal representations in $\Irr$ (resp.~$\Irr(G_n)$. 
	 Suppose $P = M \ltimes U$ and $Q = L \ltimes V$ are standard parabolic subgroups of $G$ with their standard Levi decompositions, where $Q$ is a subgroup of $P$. The functor $\ind_{L}^M \colon \Pi(L) \rightarrow \Pi(M)$ of normalized parabolic induction is defined as follows: Since $M \cap Q = L \ltimes (M \cap V)$ is a standard parabolic subgroup of $M$, for $\rho \in \Pi(L)$, we consider $\rho$ as a representation of $M \cap Q$ trivial on its unipotent radical $M \cap V$, and set $\ind_{L}^M(\rho)$ as $\ind_{M \cap Q}^M(\rho)$.
	 Let $\alpha = (n_1, \ldots, n_k)$ be a partition of $n$. Let $P = P_{\alpha} = M_{\alpha} \ltimes U_{\alpha}$ be a parabolic subgroup of $G$.  Consider $\rho_i \in \Pi(G_{n_i})$ for $1 \leq i \leq k$. Then $\rho = \rho_1 \otimes \cdots \otimes \rho_k \in \Pi(M_{\alpha})$. Define $\rho_1 \times \cdots \times \rho_k$ as $\ind_{P}^G(\rho)$.
	 
	 The functor $\ind_{L}^M$ possesses a left adjoint, the normalized Jacquet functor $\re_{M \cap V} \colon \Pi(M) \rightarrow \Pi(L)$. For $\sigma \in \Pi(M)$, $\re_{M \cap V}(\sigma)$ is the representation of $L$, which acts on the space of $M \cap V$-coinvariants of $\sigma|_{Q \cap M}$ induced by the action $\delta_{Q \cap M}^{-1/2}\sigma|_{Q \cap M}$. In this case, we denote the Jacquet functor $\re_{M \cap V}$ by $\re_{L,M}$. For $\sigma \in \Pi(M)$ and $\rho \in \Pi(L)$, there is a natural isomorphism: 
	 $$\Hom_M(\sigma, \ind_{L}^M(\rho)) \simeq \Hom_L(\re_{L,M}(\sigma), \rho).$$ 
	 Let $\beta = (\beta_1, \ldots,\beta_k)$, where $\beta_i$ is the partition of $n_i$. Let $L = M_{\beta}$, and $\pi_i \in \Pi(G_{n_i})$ for $1 \leq i \leq k$. Then we have
	 \begin{equation}\label{300}
	 	\re_{L,M}(\pi_1 \otimes \cdots \otimes \pi_k) = \re_{M_{\beta_1}, G_{n_1}}(\pi_1) \otimes \cdots \otimes \re_{M_{\beta_k}, G_{n_k}}(\pi_k).
	 \end{equation}
	 In the case where $M = M_{\alpha}$ and $L = M_{\beta}$, we indicate $\re_{L,M}$ as $\re_{\beta,\alpha}$. 
	 
	 	For a segment $\Delta = [a,b]_{\rho}$ with $\rho \in \Cusp(G_r)$, we now recall the Jacquet module of  $Q(\Delta)$. Assuming $n = r(b-a+1)$, it follows that $Q(\Delta) \in \Irr(G)$. For a partition $\alpha = (n_1,\ldots,n_k)$ of $n$, let $M = M_{\alpha}$. The analysis then splits into two cases:
	 \begin{itemize}
	 	\item If $r$ does not divide $n_i$ for some $1 \leq i \leq k$, then  $\re_{M,G}(Q(\Delta)) = 0$.
	 	\item When $r$ divides $n_i$ for each $1 \leq i \leq k$, then 
	 	\begin{equation}\label{102}
	 		\re_{M,G}(Q(\Delta)) = Q(\Delta_1) \otimes\cdots \otimes Q(\Delta_k),
	 	\end{equation}
	 	where $\Delta_i = [c_i,d_i]_{\rho}$, $d_1 = b$, $d_{i+1} = c_i -1$, $1 \leq i \leq k-1$, and $r(d_i-c_i+1) = n_i$, $1 \leq i \leq k$.
	 \end{itemize}
	 
	 \subsection{The modular character $\delta_{P}$}\label{5}
	 For all $x \in \D^{\times}$, set
	 $|x|_{\D} = |\Ne_{\D/\Fe}(x)|_{\Fe}^2 = \nu^2(x)$. 
	  Let $P_{n_1,\ldots,n_r}$ be the group of block upper triangular matrices corresponding to the tuple $(n_1,\ldots,n_r)$, with unipotent radical $N_{n_1,\ldots,n_r}$. The modulus function of the group $P_{n_1,\ldots,n_r}$ is denoted by $\delta_{P_{n_1,\ldots,n_r}}$. Since a parabolic subgroup normalizes its unipotent radical, it defines a character of $P_{n_1,\ldots,n_r}$ through
	 ​the module of the automorphism $n \mapsto pnp^{-1}$ of $N_{n_1,\ldots,n_r}$ for $p \in P_{n_1,\ldots,n_r}$. Call this character $\delta_{N_{n_1,\ldots,n_r}}$, thus $\delta_{N_{n_1,\ldots,n_r}} = \delta_{P_{n_1,\ldots,n_r}}$. For an element $p \in P_{n_1,\ldots,n_r}$ with its Levi part equal to $\diag(g_1,g_2,\ldots,g_r)$, we have 
	 $$\delta_{P_{n_1,\ldots,n_r}}(p) = |\det(g_1)|_{\D}^{n_2+\cdots+n_r}|\det(g_2)|_{\D}^{-n_1+n_3+\cdots+n_r}\cdots|\det(g_r)|_{\D}^{-n_1-n_2-\cdots-n_{r-1}}.$$
	\subsection{The Steinberg and generalized Steinberg representations}
	In the setting over $\D$, the Steinberg representation $\ST_n(\chi)$ of $G_n$ is constructed as subrepresentation of the principal series representation induced from the character $$({\vert \cdot \vert}_{\D}^{(\frac{n-1}{2})}\chi,{\vert \cdot \vert}_{\D}^{(\frac{n-3}{2})}\chi,\ldots,{\vert \cdot \vert}_{\D}^{-(\frac{n-1}{2})}\chi)$$ on minimal Levi subgroup. It can also be realized as a quotient $Q([{\vert \cdot \vert}_{\D}^{-(\frac{n-1}{2})}\chi, {\vert \cdot \vert}_{\D}^{(\frac{n-1}{2})}\chi])$ of the principal series representation 
	$${\vert \cdot \vert}_{\D}^{-(\frac{n-1}{2})}\chi \times {\vert \cdot \vert}_{\D}^{-(\frac{n-3}{2})}\chi \times \cdots \times {\vert \cdot \vert}_{\D}^{(\frac{n-1}{2})}\chi.$$ 
	Similarly, for $\mu \in \Irr (G_1)$ with $\Dim(\mu)>1$, one can define generalized Steinberg representation $\ST_n(\mu)$ of $G_n$ as a subrepresentation of the principal series representation 
	$${\vert  \Ne_{\D/\Fe} (\cdot) \vert}_{\Fe}^{(\frac{n-1}{2})}\mu \times {\vert  \Ne_{\D/\Fe} (\cdot) \vert}_{\Fe}^{(\frac{n-3}{2})}\mu \times \cdots \times {\vert  \Ne_{\D/\Fe} (\cdot) \vert}_{\Fe}^{(\frac{-(n-1)}{2})}\mu.$$
	\subsection{Some general facts}
	The following lemmas are integral to the analysis of representations of $G$, which adopt a symplectic model, though their proofs are omitted. 
	
	\begin{lemma}\label{distinction by quotient}
		Let $\pi_1, \pi_2 \in \Pi(G)$ such that $\pi_1$ appears as a quotient of $\pi_2$. If $\pi_1$ exhibits a symplectic model, then $\pi_2$ also exhibits a symplectic model.
	\end{lemma}
	
	\begin{lemma}\label{subquotient}
		If $\pi \in \Pi(G)$ has a symplectic model, then there exists an irreducible subquotient $\tau$ of $\pi$, which has a symplectic model.
	\end{lemma}
	\begin{lemma}\upshape\cite{Verma}\textbf{.}
		If $\pi \in \Irr(G)$ has a symplectic model, then  $\widetilde{\pi}$ has a symplectic model. 
	\end{lemma}
	\begin{lemma}\upshape\cite{tadic1990induced}\textbf{.} \label{7}
		Let $\Delta_1$ and $\Delta_2$ be two segments.  The representation $Q(\Delta_1) \times Q(\Delta_2)$ is irreducible if and only if  $\Delta_1$ and $\Delta_2$ are not linked; otherwise, it is reducible with two distinct irreducible subquotients, $Q(\Delta_1, \Delta_2)$ and $Q(\Delta_1 \cup \Delta_2) \times Q(\Delta_1 \cap \Delta_2)$, each with multiplicity one.
	\end{lemma}


	\section{ORBIT ANALYSIS} 
	Our main approach to classify distinguished ladder representations relies on the geometric lemma by Bernstein and Zelevinsky \cite[Theorem $5.2$]{BZ}. This process requires an extensive analysis of the double coset space $P\backslash G_n/H_n$, with $P$ being the parabolic subgroup of $G_n$. Since $H = H_n$ is a symmetric subgroup of $G = G_n$, we utilize Offen's framework \cite{offen2017parabolic} for proofs and results presented in this section. 
	\subsection{The symmetric space}
	Consider the symmetric space 
	$$X = \{x \in G \colon x \theta(x) = I_n\}$$
	with the $G$-action $$(g,x) \mapsto g\cdot x = g x \theta(g)^{-1}.$$
	Note that $XJ$ is the space of hermitian matrices in $G$ and  $$(g\cdot x)J = g(xJ)~^t\bar{g}.$$
	Therefore, $X$ is a homogeneous $G$-space. The map $(g \mapsto g\cdot I_n) \colon G \rightarrow X$ defines an isomorphism $G/H \simeq X$ of $G$-spaces. 
	For any subgroup $Q$ of $G$ and $x \in X$, let 
	$Q_x = \{g \in Q \colon g\cdot x = x\}$ be the stabilizer in $Q$ of $x$. Then 
	we observe that for $x \in X$, the automorphism $\theta_x$ on $G$ defined by $g \mapsto x\theta(g)x^{-1}$ is an involution and 
	$Q_x = (Q \cap \theta_x(Q))^{\theta_x}$.
	\subsection{Minimal Parabolic orbits in $X$}  We initiate our study by considering the minimal parabolic orbit $P_0 = M_0 U_0$ in $X$. 
	Since the map $(P_0 \cdot x \mapsto P_0 x P_0) \colon P_0 \backslash X \rightarrow P_0 \backslash G / P_0$ is well defined, Bruhat decomposition yields a map $\mathcal{I}_{M_0}$ from $P_0 \backslash X$ to $W$. Note that the involution $\theta$ satisfy the property $\theta(P_0 x P_0) = (P_0 x P_0)^{-1}$. So, each element in the image of $\mathcal{I}_{M_0}$ is a twisted involution. The proof of the following lemma is based on the proofs of \cite[Propositions $6.6$ and $6.8$]{helminck1993rationality} and \cite[Lemma $4.1.1$ and  Proposition $4.1.1$]{lapid2003periods}.

	\begin{lemma}\label{105}
		For each $x \in X$, the set $P_0\cdot x \cap N_G(T)$ is non-empty. Moreover, the map $(P_0\cdot x \mapsto P_0\cdot x \cap N_G(T)) \colon P_0 \backslash X \rightarrow M_0 \backslash (X \cap N_G(T))$ is a bijection. 
	\end{lemma}
\begin{proof}
	Let $x \in X$ be an element of the double coset $U_0 n U_0$, where $n \in N_G(T)$ associated with an element $w \in W$. Let $U_w = U_0 \cap n^{-1}U_0 n$ and $U_w' = U_0 \cap n^{-1}U_0^{-} n$, where $U_0^{-}$ is the unipotent radical opposite to $U_0$. Then there is a unique decomposition $x = unv$, where $u \in U_0$ and $v \in U_w'$. 
	We now show that $n$ belongs to the $P_0$-orbit of $x$. Since $x \in X$, we have 
	$$x = unv = \theta(v)^{-1} \theta(n)^{-1} \theta(u)^{-1}.$$    
	It follows that $\theta(n)^{-1} = n$ and $\theta(u)^{-1} = sv$, where $s \in U_w$. Therefore, under these conditions, we can conclude $x = \theta(v)^{-1} n s v$, which proves that $ns$ lies in the $P_0$-orbits of $x$. 
	Since $ns \in X$ and $\theta(n)^{-1} = n$, we have $s^{-1}n^{-1}\theta(s)^{-1}n = 1$. The map $u \mapsto n^{-1}\theta(u)n$ stabilizes $U_w$ and defines an order two Galois action. Moreover, for this action, $s^{-1}$ defines a cocycle with values in $U_w$. Since the cohomology of a unipotent group is trivial, there exists $u \in U_w$ such that $s = n^{-1}\theta(u)nu^{-1}$, and hence $ns = \theta(u)nu^{-1}$, as required.                     
	
	Since $P_0\cdot x \cap N_G(T)$ is non-empty for any $x \in X$, it remains to show that $P_0\cdot x \cap N_G(T)$ is a single $M_0$-orbit. Consider $n_1, n_2  \in X \cap N_G(T)$ such that $n_2 = p n_1 \theta(p)^{-1}$, where $p = um$ with $u \in U_0$ and $m \in M_0$. Then $n_2 = u(m n_1 \theta(m)^{-1})\theta(u)^{-1}$ and by the uniqueness property of Bruhat decomposition, we conclude that $n_2 = m n_1 \theta(m)^{-1}$, as required.   
\end{proof}
	\subsection{Parabolic orbits in $X$} For the rest of the section, let $P = M \ltimes U$ be the standard parabolic subgroup of $G$ corresponding to the partition $\alpha = (n_1,n_2,\ldots,n_k)$ of $n$. Then we have $\theta(P) = P_{n_k,\ldots,n_1}$,  $\theta(M) = w_nMw_n^{-1} = M_{n_k,\ldots,n_1}$. By following a similar argument as in the minimal parabolic case, the association $\mathcal{I}_M ( P\cdot x) = w$ defines a map
	\begin{equation}\label{101}
		\mathcal{I}_M \colon P\backslash X \rightarrow {_MW}_{\theta(M)} \cap \mathfrak{J}_0(\theta).
	\end{equation} 
	  \noindent
	  with the relation $Px\theta(P) = Pw\theta(P)$.
	  Fix $x \in X$, and let $\mathcal{I}_M (P \cdot x) = w$ and $L = M(w) = M \cap w \theta(M) w^{-1}$.
	 Note that by $\S$$\ref{bruhat}$,
	$L$ is a standard Levi subgroup of $M$ with the condition $L = w \theta(L)w^{-1}$. 
	Moreover, 
	$P \cap w \theta(P)w^{-1} = L \ltimes Z$,
	where $$Z = (M \cap w\theta(U)w^{-1})(U \cap w\theta(M)w^{-1})(U \cap w\theta(U)w^{-1})$$ is unipotent radical of $P \cap w \theta(P)w^{-1}$ with the condition $Z = w\theta(Z)w^{-1}$. For the $L$-orbits in $X$, we have the following lemma whose proof follows by the same technique as \cite[Lemma $3.2$ and $3.3$]{offen2017parabolic} with Lemma $\ref{105}$, so we omit the proof here.
	\begin{lemma}\label{l-orbit}
		With the notation above, $P\cdot x \cap Lw$ is non-empty; it yields a unique $L$-orbit in $X$. Moreover, for any element $y \in P\cdot x \cap Lw$, $P_y$ decomposes as the semidirect product of its Levi subgroup $L_y$ and the unipotent radical $Z_y$, where $\pr_M(Z_y) = M \cap w \theta(M) w^{-1}$ stands as a normal subgroup of $\pr_M(P_y)$.
	\end{lemma}

	\subsection{Admissible orbits in $X$}\label{admissible orbit}
	\begin{defi}
		An orbit $P \cdot x$  is called $M$-admissible if $M = w \theta(M) w^{-1},$ where $w = \mathcal{I}_M(P \cdot x)$.
	\end{defi}
For the proof of the following lemma, see \cite[Lemma 6.1]{offen2017parabolic}.
\begin{lemma}\label{100}
	An orbit $P \cdot x$ is called $M$-admissible if and only if $x \in U N_{G,\theta}(M)\theta(U)$. 
\end{lemma}
 Utilizing Lemma \ref{l-orbit}, we can now conclude as follows.
\begin{cor}
	There exists a bijective map defined by $(P \cdot x \mapsto P \cdot x \cap N_{G,\theta}(M))$ between $M$-admissible orbits in $P \backslash X$ to $M \backslash (X \cap N_{G,\theta}(M))$. Moreover, for $x \in X \cap N_{G,\theta}(M)$ we have $P_x = M_x \ltimes U_x$.	
\end{cor} 
For Levi subgroups $L$ of $M$, the $P$-orbits in $X$ can be analyzed in terms of specific $L$-admissible orbits.
 To be more precise, we correlate the $P$-orbit $P \cdot x$ with a particular $M(w)$-admissible orbit corresponding to $w = \mathcal{I}_M (P \cdot x)$. Thus, our initial step involves delineating the pertinent data for $M$-admissible orbits.
  
  Let 
 $S_2[\alpha] = \{\tau \in S_k \colon \tau^2 = e,~  n_{\tau(i)} = n_i, 1 \leq i \leq k\}$. The $M$-admissible orbits correspond one-to-one with $S_2[\alpha]$. 
 For the $M$-admissible $P$-orbit $P\cdot x$ in $P\backslash X$ corresponding to $\tau \in S_2[\alpha]$, we can choose a natural orbit representative $x = x_{M,\tau} \in P\cdot x \cap N_{G,\theta}(M)$ as follows: The matrix $x J_n$ has $J_{n_i}$ on its $(\tau(i),i)$-block and $0$ otherwise. It is easy to see that $M_x$ consists of elements $\diag(g_1, \ldots,g_k)$ such that 
 \begin{equation}\label{107}
 	\left\{ \begin{array}{rcl}
 		g_i J~ ^t\bar{g_i}  = J \hspace{0.8cm} & \mbox{for} & \tau(i) = i \\
 		g_{\tau(i)} = J~ ^t\bar{g_i}^{-1} J & \mbox{for}
 		& \tau(i) \not= i   
 	\end{array}. \right.
 \end{equation}
Therefore, $M_x$ is isomorphic to 
$$\prod_{i<\tau(i)} G_{n_i} \times \prod_{i=\tau(i)} H_{n_i}.$$ 
 It is easy to obtain the subsequent computation of modular characters through a routine calculation. Therefore, we omit the proof here. For $m = \diag(g_1,\ldots,g_k) \in M_x$, we have
    $$(\delta_{P_{x}} \delta_P^{-1/2}) (m) =  \prod_{i<\tau(i)}|\Ne_{\D/\Fe}(g_i)|_{\Fe}.$$

The following is an instance of admissible orbit prototypes.
\begin{exa} 
	Let $k = s + 2t$, and let $\tau(i) = i$ if $i = t+1,\ldots,t+s$ and $k+1-i$ otherwise. 
	Then $\alpha$ takes the form $(n_1,\ldots,n_t,n_{(t+1)},\ldots,n_{(t+s)},n_t,\ldots,n_1)$. Assume $$x = \diag(I_v,\ldots,J_{(n_{(t+1)},\ldots,n_{(t+s)})}J_y^{-1},I_v) = 
	\begin{pmatrix}
		&  &    J_v \\
		&  J_{(n_{(t+1)},\ldots,n_{(t+s)})} &  \\ 
		J_v   &  & 
	\end{pmatrix} J_n^{-1} \in X,
	$$
	where $J_{(n_{(t+1)},\ldots,n_{(t+s)})} = \diag(J_{n_{(t+1)}},\ldots,J_{n_{(t+s)}})$, $v = n_1 + \cdots+n_t$, and $y = n_{(t+1)} +\cdots+n_{(t+s)}$.
	Then $xJ_n$ is a hermitian matrix in $N_{G,\theta}(M)$ and hence $P\cdot x$ is $M$-admissible. 
	For each $m \in \mathbb{N}$, there is an involution on $G_{m}$ defined by $g \mapsto g^*= J_{m}~ ^t\bar{g}^{-1} J_{m}^{-1}$. We obtain that $M_x$ consists of the matrices $m = \diag(g_1,\ldots,g_t,g_{(t+1)},\ldots,g_{(t+s)}, g_t^*,\ldots,g_1^*)$, where $g_i \in G_{n_i}$, $1 \leq i \leq t$ and  $g_i \in H_{n_i}$, $t+1 \leq i \leq t+s$,
	and $$(\delta_{P_{x}} \delta_P^{-1/2}) (m) =  \prod_{i=1}^t|\Ne_{\D/\Fe}(g_i)|_{\Fe}.$$
\end{exa}
	\subsection{General orbits} \label{general orbits} Fix a $P$-orbit $P \cdot x$ in $P\backslash X$ such that $\mathcal{I}_M(P \cdot x) = w$. Let $L = M(w)$ be the standard Levi subgroup of $M$ corresponding to $w$. For $\beta = (\beta_1,\ldots,\beta_k)$ where $\beta_i = (m_{i,1},\ldots,m_{i,k_i})$ is a partition of $n_i$, let $L = M_{\beta}$. 
	We now consider the set of indices 
	$$\mathfrak{I} = \{(i,j) \colon 1 \leq i \leq k,~ 1 \leq j \leq k_i\}$$
	and define lexicographic order $(i_1,j_1)  \prec (i_2,j_2)$ if either $i_1 < i_2$ or $i_1 = i_2$ and $j_1 < j_2$. We further define the partial order $(i_1,j_1)  \ll (i_2,j_2)$ if $i_1 < i_2$.
	
	Note that by (\ref{101}) with Lemma \ref{l-orbit},  $P\cdot x \cap Lw$ is a unique $L$-admissible orbit. Moreover, for a standard parabolic subgroup $Q = P_{\beta}$ of $G$ with Levi subgroup $L$ and $x \in P\cdot x \cap Lw$, we obtain $M_x = L_x$ and $P_x = Q_x$. Thus, we can employ $\S$\ref{admissible orbit} by replacing $M$ with $L$.
	 
	By considering $S_2[\beta]$ as a set of involutions on $\mathfrak{I}$, we identify $(\mathfrak{I}, \prec)$ with the linearly ordered set $\{1,2,\ldots,|\mathfrak{I}|\}$.
	Assume $\tau \in S_2[\beta]$ the involution corresponding to $w = \mathcal{I}_L(P \cdot x)$. As $w \in  {_MW}_{\theta(M)}$, $\tau$ must satisfy 
	\begin{equation}\label{104}
		\tau(i,j+1) \ll \tau(i,j),~1 \leq i \leq k,~ 1 \leq j \leq k_i-1.
	\end{equation} 
due to additional restrictions. 
It follows that for each $i$, atmost one $j$ satisfies $\tau(i,j) = (i,j)$.
	\subsection{Orbit analysis for maximal parabolic subgroups}\label{43}  We now compute $H$-orbits in  $P_{k,n-k} \backslash G$, where $k \leq n-k$. Let $V$ be a $n$-dimensional Hermitian right $\D$-vector space having the basis $\{e_1,\ldots,e_n\}$ with $(e_i,e_{n-i+1}) = 1$ otherwise $0$. 
	Let $\mathcal{X}$ be the set of all $k$-dimensional $\D$-subspaces of $V$. As the action of $G$ on $V$ induces the transitive action on $\mathcal{X}$, so the stabilizer of an element $W = \langle e_1,e_2,\ldots,e_k\rangle$ in $\mathcal{X}$ is parabolic subgroup $P_{k,n-k}$ with $\mathcal{X} \simeq G/P_{k,n-k}$. 
	
	For $W_1, ~W_2 \in \mathcal{X}$, let $J_1$ and $J_2$ be the restrictions of $J$ to $W_1$ and $W_2$, respectively. It follows from Witt's theorem \cite{MVW} that $W_1$ and $W_2$ are conjugate by a symplectic endomorphism if and only if $\Dim \Rad J_1 = \Dim \Rad J_2$. Thus, $\mathcal{X} = \cup_{r=0}^{k} O_r$, where 
	$O_r = \{W \in \mathcal{X} | \Dim \Rad J|_{W} = r\}$.

	It is easy to see that the stabilizer of $T_k = \langle e_1,e_2,\ldots,e_k\rangle$ inside $O_k$ in $H$ is
	\begin{multline*}
		H_{k,k} = \Bigg\{
		\begin{pmatrix}
			a & b & c\\
			0& d & e \\ 
			0 & 0 & f 
		\end{pmatrix} \Bigg{|}~ a, f \in G_k,~d \in H_{n-2k},~ b \in M_{k \times (n-2k)}(\D),~ c\in M_{k \times k}(\D),~\\ e \in M_{(n-2k) \times k}(\D);~   a = J~^t \overline{f}^{-1}J,~ b~J~^t \overline{d} + a~J~^t \overline{e} = 0,~ \text{and}~ c~J~^t \overline{a} + b~J~^t \overline{b} + a~J~^t \overline{c} = 0 \Bigg\},
	\end{multline*}
	which is the parabolic subgroup of $H$ with Levi decomposition $H_{k,k} = M_{k,k} U_{k,k}$, where  
	$$M_{k,k} = \Bigg\{
	\begin{pmatrix}
		a & 0 & 0 \\
		0 & d & 0 \\ 
		0 & 0 & a^* 
	\end{pmatrix} \Bigg{|} a \in G_k,~d \in H_{n-2k}; a^* = J~^t \overline{a}^{-1}J \Bigg\} \simeq \Delta (G_k \times G_k) \times H_{n-2k},$$
\noindent
where $\Delta (G_k \times G_k) \simeq \{(a ,~ ^t\overline{a}^{-1}) ~|~ a \in G_k\}$ and $U_{k,k}$ is the unipotent subgroup inside $H$.
  
	Now, consider $T_r = \langle(e_1, e_2, \ldots, e_r,e_{r+1}+ e_{n-r},\ldots,e_{k-1}+e_{n-k+2}, e_k+e_{n-k+1})\rangle$ inside $O_r$, where $0 \leq r \leq (k-1)$.
	Clearly, the stabilizer $H_{k,r}$ of $T_r$ in $H$  takes the form $M_{k,r}U_{k,r}$, where  
	$$M_{k,r} \simeq \Delta (G_r \times G_r) \times H_{k-r} \times H_{n-k-r}$$
	with $\Delta (G_r \times G_r) \simeq \{(g ,~ ^t\overline{g}^{-1}) ~|~ g \in G_r\}$, naturally embeds into the Levi subgroup $G_k \times G_{n-k}$ corresponding to the parabolic subgroup $P_{k,n-k}$ in $G$.

	\begin{prop}\label{44} With the above notation, for $0 \leq r \leq k$,  
		the subgroup $H_{k,r}$ of $H$, which  stabilizes the subspaces $T_{r}$ of $V$ takes the form 
		$M_{k,r} \cdot U_{k,r},$
		where $M_{k,r} \simeq \Delta (G_r \times G_r) \times H_{k-r} \times H_{n-k-r}$ and $U_{k,r}$ is the unipotent subgroup inside $H$. 
		Furthermore, if $\delta_{k,r}$ is the modular character of $H_{k,r}$, then for $m = \diag (g, h_1, h_2,~^t\overline{g}^{-1}) \in M_{k,r}$, we have 
		$$\delta_{k,r}^{-1}~\delta_{P_{k,n-k}}^{1/2} (m) = |\Ne_{\D|\Fe}(g)|_{\Fe}^{-(n-2r+1)}.$$
	\end{prop}
	\section{Geometric lemma} 
	This section shows how certain discrete series representations lack a symplectic model. We further explore the symplectic model for certain induced representations, utilizing the insights from open and closed orbits within the filtration outlined in $\S$$\ref{geometric lemma}$.
	\subsection{Good orbits}\label{geometric lemma}
	Let $P = P_{\alpha} = M \ltimes U$ be a standard parabolic subgroup of $G$ associated with the partition $\alpha = (n_1,\ldots,n_k)$ of $n$. For $\sigma \in \Pi(M)$, let $\ind_P^G(\sigma)$ be a representation of $G$ with representation space $\mathcal{V}$. It follows from \cite[Theorem $5.2$]{BZ}, that there exists an ordering $y_1, \ldots,y_m$ of the double coset representatives, ensuring that 
	$$\mathcal{O}_i = \cup_{j=1}^i P y_j H$$
	is open in $G$ for all $i = 1,\ldots,m$. Define 
	$$\mathcal{V}_i = \{f \in \mathcal{V} \colon \Supp (f) \subseteq \mathcal{O}_i\}.$$
	Then 
	$$\{0\} =  \mathcal{V}_0 \subseteq \mathcal{V}_1 \subseteq
	\cdots \subseteq \mathcal{V}_m = \mathcal{V}$$ is an $H$-invariant filtration of $\mathcal{V}$. We now describe  the quotients $\mathcal{V}_i/ \mathcal{V}_{i-1}$ more explicitly.	
	
	Let $x_i = y_i \cdot e$ and $P_i = y_i^{-1}Py_i \cap H = y_i^{-1} P_{x_i}y_i$. Let $(\delta_P^{1/2}\sigma|_{P_{x_i}})^{y_i}$ be the representation of $P_i$ obtained from $\delta_P^{1/2}\sigma|_{P_{x_i}}$ by $y_i$-conjugation and $\ind_{P_i}^H$ be the non-normalized induction with compact support. Then 
	\begin{equation}\label{500}
		\mathcal{V}_i/ \mathcal{V}_{i-1} \simeq \ind_{P_i}^H((\delta_P^{1/2}\sigma|_{P_{x_i}})^{y_i}).
	\end{equation}
	\begin{defi}\label{Good orbit for}  We say that an orbit $P \cdot x_i$ is good orbit for $\sigma$ if 
		$\Hom_H(\mathcal{V}_i/ \mathcal{V}_{i-1}, \mathbb{C}) \not= 0.$
	\end{defi}
        For the rest of this subsection, fix  $y = y_i$,
        $x = x_i \in X$, $w = \mathcal{I}_M(P\cdot x)$ and $L = M(w)$. By Lemma \ref{l-orbit}, we can choose $y$ such that $x \in Lw$. Under this assumption, we define $Q$ as the standard parabolic subgroup $L \ltimes V$ of $G$, with $L$ as the standard Levi subgroup and $V$ as the unipotent radical. The subsequent proposition follows the same technique discussed in \cite[Proposition $4.1$]{offen2017parabolic}. 
	 \begin{prop}\label{50}
	 	With the above notation,  
	 	 $$\Hom_H(\ind_{y^{-1}P_xy}^H((\delta_P^{1/2}\sigma|_{P_{x}})^{y}), \mathbb{C}) \simeq \Hom_{L_x}(\re_{L,M}(\sigma), \delta_{Q_{x}} \delta_Q^{-1/2}).$$
	 	 Furthermore, the orbit $P \cdot x$ is good orbit for $\sigma$ if and only if
	 		$\Hom_{L_x}(\re_{L,M}(\sigma), \delta_{Q_{x}} \delta_Q^{-1/2}) \not= 0$.
	 \end{prop}
 \begin{cor}\label{42}
 	With the notation in $\S$\ref{43}, let $\lambda_1 \in \Irr(G_k)$ and $\lambda_2 \in \Irr(G_{n-k})$. Let $\lambda$ be  the representation of $P = P_{k,n-k}$ obtained by extending $\lambda_1 \otimes \lambda_2$ trivially to $P_{k,n-k}$. Then
 	$$\Hom_{H}(\ind_{H_{k,r}}^{H}(\delta_P^{1/2}\lambda) , \mathbb{C}) \simeq \Hom_{M_{k,r}}(\re_{(r,k-r),(k)}(\lambda_1) \otimes \re_{(n-k-r,r),(n-k)}(\lambda_2), \nu).$$
 \end{cor}
 \begin{proof}
 	Let $U = U_1 \times U_2$, where $U_1 = U_{r,k-r}$ and $U_2 = U_{n-k-r,r}$ denote the unipotent subgroups of $G_k$ and $G_{n-k}$, respectively. By Frobenius reciprocity, we obtain
 	\begin{equation}\label{200}
 	\Hom_{H}(\ind_{H_{k,r}}^{H}(\delta_P^{1/2}\lambda) , \mathbb{C}) \simeq \Hom_{H_{k,r}} (\delta_{k,r}^{-1}~\delta_P^{1/2} \lambda_1 \otimes \lambda_2, \mathbb{C}).
 	\end{equation}
 	Clearly, the space on the right hand side of $(\ref{200})$ is equivalent to $\Hom_{M_{k,r}U}(\nu^{-(n-2r+1)}\lambda_1 \otimes \lambda_2, \mathbb{C})$.
 	With the normalized Jacquet functor being left adjoint to normalized induction \cite[Proposition $1.9$(b)]{BZ}, we get
 	\begin{equation}\label{103}
 		\Hom_{M_{k,r}U}(\nu^{-(n-2r+1)}\lambda_1 \otimes \lambda_2, \mathbb{C}) \simeq \Hom_{M_{k,r}}(\re_{U}(\nu^{-(n-2r+1)}\lambda_1 \otimes \lambda_2), \delta_U^{-1/2}).
 	\end{equation}
 The space on the right hand side of $(\ref{103})$ equals
 	$\Hom_{M_{k,r}}(\nu^{-(n-2r+1)} \delta_{U_1}^{1/2} \re_{U_1}(\lambda_1) \otimes \delta_{U_2}^{1/2} \re_{U_2}(\lambda_2), \mathbb{C})$.
 	Now, for $h_1 \in H_{k-r}$ and $h_2 \in H_{n-k-r}$, we have
 	$$\delta_{U_1}
 	\begin{pmatrix}
 		g & *  \\
 		0 &  h_1  
 	\end{pmatrix}
 	= |\Ne_{\D|\Fe}(g)|_{\Fe}^{2(k-r)},~ \delta_{U_2}
 	\begin{pmatrix}
 		h_2 & *  \\
 		0 &  ^t\overline{g}^{-1}  
 	\end{pmatrix}
 	= |\Ne_{\D|\Fe}(g)|_{\Fe}^{2(n-k-r)}.$$
 	Hence, it implies
 	$$\Hom_{H}(\ind_{H_{k,r}}^{H}(\delta_P^{1/2}\lambda) , \mathbb{C}) \simeq \Hom_{M_{k,r}}(\re_{(r,k-r),(k)}(\lambda_1) \otimes \re_{(n-k-r,r),(n-k)}(\lambda_2), \nu).$$
 \end{proof}
 
    We further analyze the condition $\Hom_{L_x}(\re_{L,M}(\sigma), \delta_Q^{-1/2}\delta_{Q_{x}}) \not= 0$ for the representation \linebreak
    $\re_{L,M}(\sigma) \in \Pi(L)$.
    Considering the notation provided in $\S$\ref{general orbits}, we define $\re_{L,M}(\sigma)$ as $\rho = \otimes_{i \in (\mathfrak{I}, \prec)}\rho_i$, and we introduce $\tau \in S_2[\beta]$ as the involution acting on $\mathfrak{I}$, linked to $w = \mathcal{I}_M(P\cdot x)$ through $\S$\ref{admissible orbit}, with $L$ substituting $M$.
    Under the assumption $\rho_i \in \Irr(G_{n_i})$ whenever $\tau(i) \not= i$, $(\ref{107})$ leads us to the following implication.
	\begin{prop} \label{distinction by closed orbit}
		The space $\Hom_{L_x}(\rho, \delta_Q^{-1/2}\delta_{Q_{x}})$ is non-zero if and only if $\rho_i \simeq \nu \rho_{\tau(i)}$ for all $i \in \mathfrak{I}$ such that $i \prec \tau(i)$ and $\rho_i$ is $H_{n_i}$-distinguished if $\tau(i) = i$. 
		In particular, if  $P\cdot x$ is good orbit for $\sigma$,  then there is an irreducible component $\rho = \otimes_{i \in (\mathfrak{I}, \prec)}\rho_i \in \Irr(L)$ of $\re_{L,M}(\sigma)$ satisfying 
		$\rho_i \simeq \nu \rho_{\tau(i)}$ for all $i \in \mathfrak{I}$ such that $i \prec \tau(i)$ and $\rho_i$ is $H_{n_i}$-distinguished if $\tau(i) = i$. 
	\end{prop}

	The version of the geometric lemma explored above is frequently utilized to demonstrate the non-distinguished nature of specific induced representations. By Lemma \ref{subquotient}, its foundation rests upon the following observation.
	 
	\begin{lemma}\label{47} Let $P = M \ltimes U$ be a standard parabolic subgroup of $G$ and let $\sigma \in \Pi(M)$. If $\ind_{P}^G(\sigma)$ admits a symplectic model, then a $P$-orbit in $X$ serves as a good orbit for $\sigma$. 
	\end{lemma}

\subsection{Non-distinguished nature of $\ST_n(\chi)$ and $\ST_n(\mu)$} We now discuss the non-distinguished nature in two specific instances.
\begin{proof}[\upshape\textbf{Proof of Proposition $\ref{301}$}]
By Lemma $\ref{7}$,  it is easy to see that $\ST_n(\chi)$ sits in the following exact sequence 
$$0 \rightarrow Q(\Delta_1,\Delta_2) \rightarrow \pi = Q(\Delta_1) \times Q(\Delta_2) \rightarrow \ST_n(\chi) \rightarrow 0,$$
where $\Delta_1 = [{\vert \cdot \vert}_{\D}^{-(\frac{n-1}{2})}\chi, {\vert \cdot \vert}_{\D}^{-(\frac{n-3}{2})}\chi]$ and $\Delta_2 = [{\vert \cdot \vert}_{\D}^{-(\frac{n-5}{2})}\chi, {\vert \cdot \vert}_{\D}^{(\frac{n-1}{2})}\chi$]. As $\ST_n(\chi)$ appears as a quotient of $\pi$, the absence of distinction in $\pi$ correlates with the non-distinguished nature of $\ST_n(\chi)$. We now analyze the necessary conditions under which $\pi$ exhibits a symplectic model.
Applying Mackey theory, we find that $\pi{\vert}_{H}$
decomposes into three subquotients given by
$$\ind_{H_{2,r}}^{H}[(\delta_{P_{2,n-2}}^{1/2}Q(\Delta_1) \otimes Q(\Delta_2)){\vert}_{H_{2,r}}],~ r = 0,1,2.$$

By analyzing these three subquotients using Corollary $\ref{42}$, it is easy to see that the necessary conditions for $\pi$ to have a symplectic model are the following: 
\begin{itemize}
	\item[\upshape(1)] Both $Q(\Delta_1)$ and $Q(\Delta_2)$ must admit a symplectic model,
	\item[\upshape(2)] $\Hom_{M_{2,1}}(\re_{(1,1),(2)}(Q(\Delta_1)) \otimes \re_{(n-3,1),(n-2)}(Q(\Delta_2)), \nu) \not= 0$,
	\item[\upshape(3)] $\Hom_{M_{2,2}}(Q(\Delta_1) \otimes \re_{(n-4,2),(n-2)}(Q(\Delta_2)), \nu) \not= 0$.
\end{itemize} 
\noindent
It is evident from these conditions that $\pi$ does not exhibit a symplectic model. Therefore, $\ST_n(\chi)$ also does not exhibit a symplectic model.
\end{proof} 
\begin{rem}
	 For odd $n$, the non-distinguishness of $\ST_n(\chi)$ can also be seen by 
	{\upshape\cite[Theorem 6.6]{Verma}} if we grant ourselves that globalization theorem in {\upshape\cite{prasad2008generalised}} holds for discrete series representations. 
\end{rem}
We also demonstrate that for a segment $\Delta$, a certain class of  representations of the form $Q(\Delta)$ does not have a symplectic model, which we prove in the following lemma. 
\begin{lemma}\label{55}
	For $a < b$, let $\Delta = [a, b]_{\rho}$ be a segment in $\Sgm$, where 
	$\rho \in \Cusp(G_k)$ does not have a  symplectic model. Then $Q(\Delta)$ does not exhibit a symplectic model. 
\end{lemma}
\begin{proof} Let $n = k(b-a+1)$.
	Since $Q(\Delta)$ is realized as a quotient of a representation $\nu_{\rho}^a \rho \times Q([a+1, b]_{\rho})$ of $G$, its distinction is associated with the distinguished nature of a representation of the form $\pi = \Ind_{P_{k,n-k}}^{G}(\rho_1 \otimes \rho_2)$, where $\rho_1 \in \Irr(G_k)$ and $\rho_2 \in \Irr(G_{n-k})$.
	Utilizing Mackey theory, we determine that the restriction $\pi{\vert}_{H}$ breaks down into subquotients expressed as
	$$\ind_{H_{k,r}}^{H}[(\delta_{P_{k,n-k}}^{1/2}\rho_1 \otimes \rho_2){\vert}_{H_{k,r}}],~ 0 \leq r \leq k.$$
	
	By examining these subquotients through Corollary $\ref{42}$, it becomes clear that for $\pi$ to possess a symplectic model, the following conditions must hold:
	\begin{itemize}
		\item either both $\rho_1$ and $\rho_2$ must have a symplectic model or 
		\item $\Hom_{M_{k,r}}(\re_{(r,k-r),(k)}(\rho_1) \otimes \re_{(n-k-r,r),(n-k)}(\rho_2), \nu) \not= 0$, where $r = 1,2, \ldots, k$.
	\end{itemize}
	 \noindent
	These conditions indicate that $\nu_{\rho}^a \rho \times Q([a+1, b]_{\rho})$ lacks a symplectic model, thereby implying that $Q(\Delta)$ does not have a symplectic model by Lemma \ref{distinction by quotient}.
\end{proof}
\begin{cor}
	The generalized Steinberg representations $\ST_n(\mu)$ of $G_n$ do not exhibit a symplectic model.
\end{cor}
	\subsection{Distinction by open and closed orbits} 
	The converse of the Lemma $\ref{47}$ may not hold in all instances. Although, the geometric lemma suggests distinction in two specific instances; hence the proof discussed below delineates the  hereditary nature of the symplectic model, where this distinction is related to the open orbit, following the notation from $\S$$\ref{geometric lemma}$.
	\begin{proof}[\upshape\textbf{Proof of Theorem $\ref{48}$}] 
        For a partition $\alpha = (n_1,\ldots,n_k)$ of $n$,
		assume that  $z = J_{(n_1,\ldots,n_k)}J_n^{-1}$. Therefore, $P \cap \theta_z(P) = M$, the orbit $P \cdot z$ is open in $X$, and 
		$$M_{z} = \{\diag(h_1,\ldots,h_k) \colon h_i \in H_{n_i},~ 1 \leq i \leq k \}.$$
		Since $\sigma_i$ admits a symplectic model for each $i$, there exist a linear form $f_i$ which lies in $\Hom_{H_{n_i}}(\sigma_i, \mathbb{C})$. 
		
		Assume that $\sigma = \sigma_1 \otimes \cdots \otimes \sigma_k$ and $f = f_1 \otimes \cdots \otimes f_k \in \Hom_{M_{z}}(\sigma, \mathbb{C})$. Then  for $\eta \in G$ with $\eta \cdot I_{n} = z$, the integral 
		$$\widetilde{f}(\phi) = \int_{H \cap \eta^{-1}M_{z} \eta \backslash H} \phi(\eta h) dh$$
		defines a non-zero linear form $\widetilde{f} \in \Hom_H(\mathcal{V}_1, \mathbb{C})$. There is no assurance  of linear form $\widetilde{f}$ for the extension to an $H$-invariant linear form on $\ind_{P}^G(\sigma)$, 
		but it lies in a holomorphic family with the ability to extend meromorphically.
		For $\mu = (\mu_1,\ldots,\mu_k) \in \mathbb{C}^k$, we define the representation $\sigma[\mu]$ acting on the space of $\sigma$ by $$\sigma[\mu](\diag(g_1,\ldots,g_k)) = |\Ne_{\D/\Fe}(g_1)|_{\Fe}^{\mu_1}\cdots|\Ne_{\D/\Fe}(g_k)|_{\Fe}^{\mu_k} \sigma(\diag(g_1,\ldots,g_k)).$$
		Then Theorem $2.8$ in \cite{BD08} assures the presence of a non-zero meromorphic function \linebreak$(\mu \mapsto f_{\mu}) \colon \mathbb{C}^k \rightarrow \mathcal{V}^*$ that satisfies $f_{\mu} \in \Hom_H(\ind_P^G(\sigma[\mu]), \mathbb{C})$ whenever holomorphic at $\mu$. Therefore, when evaluating the leading term at $\mu = 0$ along a complex line passing through zero in a generic direction, it furnishes a nontrivial element of $\Hom_H(\ind_P^G(\sigma), \mathbb{C})$.  
	\end{proof}

    In the context of the closed orbit $P\cdot I_n$, we have the following theorem by the geometric lemma, a consequence of Proposition \ref{distinction by closed orbit}.
	
	\begin{theorem}\label{49}
		Let $\sigma_i \in \Irr$ for $1 \leq i \leq k$, and  $\rho_i \in \Pi$ for $1 \leq i \leq t$ (allow the case $t=0$) with each $\rho_i$ admitting a symplectic model. Then 
		$\nu \sigma_1 \times \cdots \times \nu \sigma_k \times \rho_1 \times \cdots \times \rho_t \times \sigma_k \times \cdots \times \sigma_1$
		also admits a symplectic model.
	\end{theorem}
	\begin{proof}
		Since each $\rho_i$ admits a symplectic model, $\rho = \rho_1 \times \cdots \times \rho_t$ also admits a symplectic model by Theorem \ref{48}. Let $P$ be a parabolic subgroup of $G$, with Levi subgroup $M$ associated with the representation $\sigma = \nu \sigma_1 \otimes \cdots \otimes \nu \sigma_k \otimes \rho \otimes \sigma_k \otimes \cdots \otimes \sigma_1$. Then $\ind_P^{G}(\sigma) \in \Pi(G)$ and closed orbit $P\cdot I_n$ is good orbit for $\sigma$ by Proposition \ref{distinction by closed orbit}. By Definition $\ref{Good orbit for}$, we have $\Hom_H(\mathcal{V}/ \mathcal{V}_{m-1}, \mathbb{C}) \not= 0$. Hence, $\ind_P^{G}(\sigma)$ admits a symplectic model by Lemma \ref{distinction by quotient}.
	\end{proof}

	\section{Consequences of good orbits}\label{59}
	This section begins by defining notations related to the multisets of integer segments, followed by demonstrating a connection between these multisets and the symplectic model, which is a consequence of good orbits. 
	
	\subsection{Symplectic multisets}
	A segment of integers refers to a set of the form $[a, b] = $ \linebreak$\{a, a + 1,\ldots, b\}$ for some integers $a \leq b$. The set $\Sgm_{\mathbb{Z}}$ represents all such integer segments. Define
	$$\nu[a,b] = [a+1,b+1], \hspace{.5cm} [a,b] \in \Sgm_{\mathbb{Z}}$$
	and consider the following relation on $\Sgm_{\mathbb{Z}}$. Let $[a_1,b_1]$, $[a_2,b_2]$ be two segments. Then $[a_1,b_1]$ precedes $[a_2,b_2]$ (or $[a_1,b_1] \prec [a_2,b_2]$ ) if $a_1<a_2$, $b_1 < b_2$, and $b_1 \geq a_2- 1$.
	
	We refer to a decomposition of $[a, b] \in \Sgm_{\mathbb{Z}}$ as $k$-tuple of segments $([a_1, b_1],\ldots, [a_k, b_k]) \in
	\Sgm_{\mathbb{Z}}^k$, $k \in \mathbb{N}$, such that $b_1 = b$, $a_k = a$, and $b_{i+1} = a_i-1$, $i = 1,\ldots, k-1$. When $k = 1$, the decomposition is called trivial. Let $\mathcal{O}_{\mathbb{Z}}$ denote the set of multisets of segments of integers. Then a multiset $\mathfrak{m} = \{\Delta_1,\ldots,\Delta_k\}$ is ordered in standard form if $\Delta_i \nprec \Delta_j$ for all $i < j$. Given the ordered multiset $\mathfrak{m} = \{\Delta_1,\ldots,\Delta_k\} \in \mathcal{O}_{\mathbb{Z}}$,
	a decomposition of $\mathfrak{m}$ implies a corresponding decomposition of each individual $\Delta_i$. When the decomposition of each $\Delta_i$ is trivial, the decomposition of $\mathfrak{m}$ is trivial.
	
	One can also index a decomposition of an ordered multi-set as follows. Let $(\Delta_{i,1},\ldots,\Delta_{i,k_i})$ be the decomposition of $\Delta_i$ for $i = 1,\ldots,k$. Let $(\mathfrak{I}, \prec)$ be the linearly ordered set
	$$\mathfrak{I} = \{(i, j) \colon i = 1,\ldots, k, j = 1,\ldots, k_i\}$$
	with the lexicographic order $(i_1, j_1) \prec (i_2, j_2)$ if either $i_1 < i_2$ or $i_1 = i_2$ and $j_1 < j_2$. We further define the partial order $(i_1, j_1) \ll (i_2, j_2)$ if $i_1 < i_2$. Therefore, breaking down an ordered multiset $\mathfrak{m} = \{\Delta_1,\ldots,\Delta_k\} \in \mathcal{O}_{\mathbb{Z}}$ yields a new ordered multiset    
	$\{\Delta_i \colon i \in (\mathfrak{I}, \prec) \}$.

	\begin{defi}\label{10}
		Let  $\mathfrak{m} = \{\Delta_1,\ldots,\Delta_k\} \in \mathcal{O}_{\mathbb{Z}}$ be an ordered multiset. Then an ordered decomposition $\{\Delta_{(i,j)} \colon (i,j) \in (\mathfrak{I}, \prec) \}$ of the order $\{\Delta_1,\ldots,\Delta_k\}$ of $\mathfrak{m}$ is said to be good decomposition for $\mathfrak{m}$ if there is an involution $\tau$ on $\mathfrak{I}$ that  satisfies the following properties:
		\begin{enumerate}
			\item[\upshape(1)] $\tau(i,j+1) \ll \tau(i,j),~i = 1,\ldots,k,~j = 1,\ldots,k_i-1;$
			\item[\upshape(2)] $\tau(i,j) \not= (i,j)$, $(i,j) \in \mathfrak{I}$;
			\item[\upshape(3)] $\Delta_{(i,j)} = \nu \Delta_{\tau(i,j)}$ whenever $(i,j) \prec \tau(i,j)$.
		\end{enumerate}
	\end{defi}
	\begin{defi}\label{58}
		A multiset $\mathfrak{m} \in \mathcal{O}_{\mathbb{Z}}$ is called symplectic if it admits a good decomposition for each standard order of $\mathfrak{m}$.
	\end{defi}

	 For $\mathfrak{m} = \{\Delta_1,\ldots,\Delta_k\} \in \mathcal{O}_{\mathbb{Z}}$, define $\nu \mathfrak{m} = \{\nu\Delta_1,\ldots,\nu\Delta_k\}$.
	 The symplectic multisets exhibit the following property, as established by Mitra et al., {\upshape\cite[Proposition $8.7$]{mitra2017klyachko}}.
	\begin{theorem}\label{54} Let $\mathfrak{m} \in \mathcal{O}_{\mathbb{Z}}$ be  a set. If $\mathfrak{m}$ is symplectic, then $\mathfrak{m}$ takes the form $\mathfrak{m'} + \nu \mathfrak{m'}$ for some $\mathfrak{m'} \in \mathcal{O}_{\mathbb{Z}}$. 
	\end{theorem}
	
	\begin{rem}
		The bijection map $[a,b] \mapsto [a,b]_{\rho}$ from segments of integers to segments in supercuspidal line of $\rho$  induces a bijection between multisets from $\mathcal{O}_{\mathbb{Z}}$ to $\mathcal{O}_{\rho}$. This correspondence is termed $\rho$-labeling, with its inverse operation referred to as unlabeling.  For the rest of the article, we take the following convention. A property defined on multisets in $\mathcal{O}_{\mathbb{Z}}$ is said to be satisfied by a multiset $\mathfrak{m} \in \mathcal{O}_{\rho}$ if its unlabeling satisfies this property.
	\end{rem}

	\subsection{Connection between symplectic model and symplectic multiset}\label{62} 
	Fix $\rho \in \Cusp (G_{r})$ without symplectic model. Let $\mathfrak{m} = \{\Delta_1,\ldots,\Delta_k\} \in \mathcal{O}_{\rho}$ be a set, where $\Delta_i = [a_i,b_i]_{\rho}$, $i = 1,\ldots,k$. We now prove that if
	$\lambda(\mathfrak{m}) = Q(\Delta_1) \times\cdots \times Q(\Delta_k)$ admits a symplectic model, then the unlabeling of $\mathfrak{m}$ is symplectic in the sense of Definition \ref{58}.
	To prove this, we recall the Jacquet module of the representation $\sigma = Q(\Delta_1) \otimes\cdots \otimes Q(\Delta_k)$. 

	Let $M = M_{\alpha}$ and $L = M_{\beta}$ where $\beta$ represents a refinement of a partition $\alpha$ of $n$. Note that in the case of the ordered index set $\mathfrak{I}$ and the decomposition components, we use the notation of $\S$\ref{general orbits}. 
	Let $\sigma$ be an irreducible, essentially square-integrable representation of $M$. If Jacquet module of $\sigma$ is non-zero, then by using Jacquet module of $Q(\Delta)$ provided in $\S$$\ref{303}$ together with $(\ref{300})$, we obtain
	\begin{equation}\label{60}
		\re_{L,M}(\sigma) = \otimes_{(i,j) \in (\mathfrak{I}, \prec)}Q(\Delta_{(i,j)}),
	\end{equation}
	where $\re_{M_{\beta_i},G_{n_i}}(Q(\Delta_i)) = Q(\Delta_{i,1}) \otimes\cdots \otimes Q(\Delta_{i,k_i})$ is defined by $(\ref{102})$.
	 
	 Let $M$ and $\sigma$ be as above and let $\mathcal{I}_M(P \cdot x) = w \in  {_MW}_{\theta(M)} \cap \mathfrak{J}_0(\theta)$. For each $(i,j) \in \mathfrak{I}$,  the representation $Q(\Delta_{(i,j)})$ does not have a symplectic model by Lemma \ref{55}. Therefore, by employing Proposition \ref{distinction by closed orbit} with (\ref{60}), we derive the subsequent proposition. 
	 \begin{prop}\label{61}
	 	With the above notation, the following are equivalent:
	 	\begin{enumerate}
	 		\item[\upshape(1)] $P \cdot x$ is good orbit for $\sigma$;
	 		\item[\upshape(2)] $\tau(i,j) \not= (i,j)$ for each  $(i,j) \in \mathfrak{I}$, and $Q(\Delta_{(i,j)}) \simeq \nu Q(\Delta_{\tau{(i,j)}})$ if $(i,j) \prec \tau(i,j)$.
	 	\end{enumerate}
	 \end{prop}
     It leads to the following result.
\begin{prop}\label{65}
	Let $\mathfrak{m}  = \{\Delta_1,\ldots,\Delta_t\}  \in \mathcal{O}_{\rho}$ be a set. If $\lambda(\mathfrak{m})$ admits a symplectic model, then $\mathfrak{m} = \mathfrak{m'} + \nu \mathfrak{m'}$ for some $\mathfrak{m'} \in \mathcal{O}_{\rho}$. In particular, if $Q(\mathfrak{m})$ admits a symplectic model, then $\mathfrak{m} = \mathfrak{m'} + \nu \mathfrak{m'}$ for some $\mathfrak{m'} \in \mathcal{O}_{\rho}$.
\end{prop}
\begin{proof}
	Since $\lambda(\mathfrak{m})$ has a symplectic model, the representation $\lambda(\mathfrak{m}) = Q(\Delta_1) \times \cdots \times Q(\Delta_t)$ also has a symplectic model for any standard order on $\mathfrak{m} = \{\Delta_1,\ldots,\Delta_t\}$. As a result, Lemma \ref{47} provides a good orbit for $Q(\Delta_1) \otimes \cdots \otimes Q(\Delta_t)$. It is easy to observe that Proposition \ref{61} with (\ref{104}) implies the symplectic nature of  $\mathfrak{m}$ in the sense of Definition \ref{58}. Thus, $\mathfrak{m}$ takes the form $\mathfrak{m'} + \nu \mathfrak{m'}$ for some $\mathfrak{m'} \in \mathcal{O}_{\rho}$, by Theorem $\ref{54}$.  
\end{proof}

	\section{Distinguished ladder and unitary representations}
	In this section, Theorem \ref{1} is established, presenting a family of ladder representations with a symplectic model. Utilizing the hereditary property of the symplectic model together with Theorem \ref{1}, we prove Theorem \ref{4}, which provides a family of irreducible unitary representations equipped with a symplectic model.
	
 \subsection{The representation $Q(\nu \Delta , \Delta)$} We begin with an example of ladder representation with a symplectic model.
 \begin{lemma}\label{41}
 	Let $\rho \in \Cusp$, without a symplectic model. Let $\Delta = [a, b]_{\rho}$ be a segment such that $a < b$. Then $Q(\nu \Delta , \Delta)$ admits a symplectic model.
 \end{lemma}
\begin{proof} 
	Since $Q(\nu \Delta , \Delta)$ sits in the following exact sequence
	$$0 \rightarrow  Q([a+1, b]_{\rho}) \times Q([a, b+1]_{\rho})  \rightarrow \pi = \nu Q(\Delta) \times Q(\Delta) \rightarrow Q(\nu \Delta , \Delta) \rightarrow 0,$$ 
	it follows from Theorem \ref{49} that $\pi$ has a symplectic model. Using Mackey theory with Lemma \ref{55}, we conclude that $Q([a+1, b]_{\rho}) \times Q([a, b+1]_{\rho})$ does not have a symplectic model. Therefore, $Q(\nu \Delta , \Delta)$ has a symplectic model by Lemma \ref{subquotient}. 
\end{proof}

	\subsection{Projection of the representation $\lambda(\mathfrak{m})$}\label{66}
	For the segments $\Delta_i = [a_i,b_i]_{\rho}$ in $\Cusp$, where $1 \leq i \leq t$, let $\mathfrak{m} = \{\Delta_1,\ldots,\Delta_t\} \in \mathcal{O}_{\rho}$ be a ladder. For each $1 \leq i \leq t-1$, define 
	$$\mathcal{P}_i = Q(\Delta_1) \times \cdots \times Q(\Delta_{i-1}) \times Q(\Delta_i^*) \times Q(\Delta_{i+1}^*) \times Q(\Delta_{i+2}) \times\cdots \times Q(\Delta_t),$$
	where $\Delta_i^* = [a_{i+1},b_i]_{\rho}$ and $\Delta_{i+1}^* = [a_i,b_{i+1}]_{\rho}$
	(with $\mathcal{P}_i = 0$ whenever $a_i > b_{i+1} + 1$). Then we have the following by \cite[Theorem $1$]{lapid2014determinantal}.
	\begin{lemma}\label{67}
		With the above notation,
		let $\mathcal{P}$ be the kernel of the projection $\lambda(\mathfrak{m}) \rightarrow Q(\mathfrak{m})$.
		 Then $\mathcal{P} = \sum_{i=1}^{t-1}\mathcal{P}_i$.
	\end{lemma}
  \subsection{Proof of Theorem \ref{1}}
  \begin{proof}
        It is easy to see that the condition $(2)$ and $(3)$ are equivalent. If $Q(\mathfrak{m})$ has a symplectic model, then Proposition \ref{65} implies that $\mathfrak{m} = \mathfrak{m'} + \nu \mathfrak{m'}$ for some $\mathfrak{m'} \in \mathcal{O}_{\rho}$. 
        
        Let $t=2k$ be a even number, and for every $1 \leq i \leq k$, let 
        $\Delta_{2i-1} = \nu \Delta_{2i}$, $\mu_i = Q(\Delta_{2i-1},\Delta_{2i})$, and $\mu = \mu_1 \times \cdots \times \mu_k.$
        Since each $\mu_i$ has a symplectic model by Lemma \ref{41},  Theorem \ref{48} implies that $\mu$ has a symplectic model. Note that $\mu$ appears as a quotient of $\lambda(\mathfrak{m})$. Hence, $\lambda(\mathfrak{m})$ also has a symplectic model by Lemma \ref{distinction by quotient}.
        
        Using the notation from $\S$\ref{66}, we proceed to establish the existence of a symplectic model for $Q(\mathfrak{m})$. This can be achieved by demonstrating that $\mathcal{P}$ lacks a symplectic model, which, according to Lemma \ref{67}, is equivalent to proving the absence of a symplectic model in $\mathcal{P}_i$ for each $1 \leq i \leq 2k-1$.
		Since the multiset 
		$\mathfrak{m}_i = \{\Delta_1,\ldots,\Delta_{i-1},\Delta_{i}^*,\Delta_{i+1}^*,\Delta_{i+2},\ldots,\Delta_t\} \in \mathcal{O}_{\rho}$
		 is arranged in strictly decreasing end points for all $1 \leq i \leq 2k-1$, it is a set in standard form. Hence,
		 $\mathcal{P}_i$ is isomorphic to $\lambda(\mathfrak{m}_i)$. For any $1 \leq i \leq 2k-1$, it suffices to prove by Proposition \ref{65} that $\mathfrak{m}_i$ is not of Speh type.
		 
		 By contradiction, let us assume that a multiset $\mathfrak{m}_i$ of Speh type exists for some $i$.  
		 Let $\Delta_j' = \Delta_j$ for $j \not= i,i+1$, $\Delta_i' = \Delta_{i}^*$ and $\Delta_{i+1}' = \Delta_{i+1}^*$. 
		 Since the multiset 
		 $\mathfrak{m}_i = \{\Delta_1',\ldots,\Delta_{2k}'\}$
		 is arranged in strictly decreasing end points, we have $\Delta_{2j-1}' = \nu \Delta_{2j}'$ for all 
		 $1 \leq j \leq k$. If $i$ is odd, then the inequality $a_i > a_{i+1}$ is contradicted by $a_{i+1} = a_i + 1$. If $i$ is even, then the fact $a_{i-1} > a_i > a_{i+1}$ are integers is contradicted by $a_{i-1} = a_{i+1} + 1$. Hence, the theorem follows. 
		\end{proof}

	\subsection{Proof of Theorem $\ref{4}$}\label{68} 
	\begin{proof}
	Let $\rho \in \Cusp$, which lacks a symplectic model, and let $\Delta = [a, b]_{\rho}$ represent a segment in $\Sgm$. Assume that $\delta = Q(\Delta)$ is an essentially square-integrable representation associated with $\rho$. Let $\Lambda_{s} = ((s-1)/2,(s-3)/2,\ldots,-(s-1)/2)$, for $s \in \mathbb{N}$. Set 
	$$\mathbf{I}(\delta, \Lambda_{s}) = \nu_{\delta}^{\frac{(s-1)}{2}}\delta \times \nu_{\delta}^{\frac{(s-3)}{2}}\delta \times \cdots \times \nu_{\delta}^{\frac{-(s-1)}{2}}\delta,$$
	   and express $\mathbf{I}(\delta, \Lambda_{s})$ as
		$Q([a_1,b_1]_{\rho}) \times Q([a_2, b_2]_{\rho}) \times \cdots \times Q([a_s,b_s]_{\rho}),$
		where
		\begin{align*}
			\begin{cases}
				a_1 = a+(s-1)/2, ~b_1 = b+(s-1)/2,\\
				a_{i+1} = a_i-1,~b_{i+1} = b_i-1,~ 1 \leq i \leq s-1,~\text{and}\\
				a_i + b_{s+1-i} = a+b,~ 1 \leq i \leq s.
			\end{cases}
		\end{align*}
		 Consequently, by Lemma $\ref{67}$, we obtain the subsequent exact sequence:
		$$0 \rightarrow \mathcal{P} \rightarrow \mathbf{I}(\delta, \Lambda{s}) \rightarrow \Sp(\delta,s) \rightarrow 0,$$
		where $\mathcal{P} = \sum_{i=1}^{t-1}\mathcal{P}i$ is explicitly defined in $\S$\ref{66}.
		From the above expression for $\mathbf{I}(\delta, \Lambda{s})$, it is evident that $\Sp(\delta,s)$ represents a ladder representation for odd $s$. Theorem \ref{1} indicates that $\Sp(\delta,s)$ lacks a symplectic model in this case. However, for even $s$, Theorem \ref{1} shows that $\nu_{\delta}^{1/2} \Sp(\delta,s)$ has a symplectic model, which means that $\Sp(\delta,s)$ itself must also have a symplectic model.
		
		As a result, each $\rho_i$ has a symplectic model for $1 \leq i \leq u$, and by Theorem \ref{48}, each $\lambda_i$ also has a symplectic model for $1 \leq i \leq v$. Moreover, for every $1 \leq i \leq y$ $\sigma_i$ have a symplectic model according to  \cite[Theorem 1.1]{sharma2023symplectic}. As $\tau_i$ has a symplectic model for each $1 \leq i \leq x$, it follows by Theorem \ref{48} that $\pi$ has a symplectic model.
	\end{proof}
	\begin{center}
		\large {\textbf{Acknowledgment}}
	\end{center}
	We thank Dipendra Prasad for his insightful comments in Theorem $\ref{4}$ and his suggestions on modular characters and the Steinberg representations. 
	The first author acknowledges the research support from the ``Council for Scientific and Industrial Research (Government of India)" (File No. 09/143(1005)/2019-EMR-I). The second author is partially supported by the SERB, MTR/2021/000655.
	\bibliography{SMFLUR}
	\bibliographystyle{alpha}
\end{document}